\documentclass{amsart}

\usepackage[centertags]{amsmath} 
\usepackage{amsfonts,amssymb,amsthm,amscd,latexsym,euscript,bbm,stmaryrd}
\usepackage[all]{xy}

\newtheorem {theorem}{Theorem}[section]
\newtheorem {corollary}[theorem]{Corollary}

\newtheorem {lemma}[theorem]{Lemma}
\theoremstyle{definition}
\newtheorem {definition}[theorem]{Definition}
\newtheorem {example}[theorem]{Example}
\newtheorem {remark}[theorem]{Remark}
\newtheorem {axiom}[theorem]{Axiom}
\newtheorem {question}[theorem]{Question}

\numberwithin{equation}{section}        

\def\setrefstep#1{\setcounter{enumi}{#1}}
\newcounter{zahl}%
\newenvironment{punkt}{\begin{list}{{\rm{(\roman{zahl})}}}%
    {\usecounter{zahl}%
     \setlength{\leftmargin}{0pt} \setlength{\itemindent}{4pt} \setlength{\topsep}{2pt} \setlength{\parsep}{2pt} }}%
    {\end{list}}%
\newcommand{\co}{\colon\thinspace}
\newcommand{\mc}[1]{\ensuremath{\mathcal{#1}}}

\newcommand{\toh}[1]{\ensuremath{\stackrel{#1}{\rightarrow}}}%
\newcommand{\colim}{\operatorname*{colim}}
\newcommand{\hocolim}{\operatorname*{hocolim\,}}
\newcommand{\frei}{\,\_\!\_\,}
\newcommand{\id}{\operatorname{id}}
\newcommand{\LKan}{\operatorname*{LKan}}
\newcommand{\inj}{\ensuremath{\hookrightarrow}}
\newcommand{\scup}{\operatorname*{\sqcup}}%
\newcommand{\x}{\operatorname*{\times}}%

\newcommand{\winkel}{\ulcorner\hspace{-3pt}\raisebox{1pt}{\ensuremath{\cdot}}}%
\newcommand{\pushout}{\ensuremath{\displaystyle \winkel}}%
\newcommand{\ul}[1]{\underline{#1}}
\newcommand{\ol}[1]{\overline{#1}}
\newcommand{\Ev}{\ensuremath{{\rm Ev}}}%
\newcommand{\cyl}{\ensuremath{{\rm Cyl}}}%
\newcommand{\Sp}{\ensuremath{{\it Sp}}}%
\newcommand{\symsp}{\ensuremath{\Sp^{\Sigma}}}%
\newcommand{\Hom}[3]{\ensuremath{{\rm Hom}_{#1}(#2,#3)}}%
\newcommand{\dgrm}[1]{\ensuremath{\smash{\underset{\widetilde{\hphantom{#1}}}{#1}} \mathstrut}}%
\newcommand{\Vnat}{\ensuremath{\mc{V}{\rm -nat}}}%
\newcommand{\ulvmod}{\ensuremath{\ul{\mc{V}{\rm -mod}}}}%
\newcommand{\sph}{\ensuremath{{\rm Sph}}}%
\newcommand{\vsph}{\ensuremath{{\rm Sph}^{\mc{V}}}}%
\newcommand{\sigsph}{\ensuremath{{\rm Sph}^{\Sigma}}}%
\newcommand{\stab}{\ensuremath{{\rm stab}}}%
\newcommand{\fib}{\ensuremath{{\rm fib}}}%
\newcommand{\fibfun}{\ensuremath{(\frei)^{\fib}}}%
\newcommand{\cof}{\ensuremath{{\rm cof}}}%
\newcommand{\diag}[2]{ \begin{align} \begin{split} \xymatrix{#1} \end{split} \label{#2} \end{align}}%
\newcommand{\diagr}[1]{ \begin{equation*} \xymatrix{#1} \end{equation*}}%

\usepackage{hyperref} 

\begin{document}

\SelectTips{cm}{10}

\title  [$L$-stable functors]
        {$L$-stable functors}
\author{Georg Biedermann}

\address{Department of Mathematics, Middlesex College, The University of Western Ontario, London, Ontario N6A 5B7, Canada}

\email{gbiederm@uwo.ca}



\date{\today}

\begin{abstract}
We generalize and greatly simplify the approach of Lydakis and Dundas-R\"ondigs-\O stv\ae r to construct an $L$-stable model structure for small functors from a closed symmetric monoidal model category \mc{V} to a \mc{V}-model category \mc{M}, where $L$ is a small cofibrant object of \mc{V}. For the special case $\mc{V}=\mc{M}=\mc{S}_*$ pointed simplicial sets and $L=S^1$ this is the classical case of linear functors and has been described as the first stage of the Goodwillie tower of a homotopy functor. 
We show, that our various model structures are compatible with a closed symmetric monoidal product on small functors. 
We compare them with other $L$-stabilizations described by Hovey, Jardine and others. This gives a particularly easy construction of the classical and the motivic stable homotopy category with the correct smash product. We establish the monoid axiom under certain conditions.
\end{abstract}

\maketitle
\tableofcontents

\section{Introduction}

Stable homotopy theory is concerned with solving homotopy theoretical problems up to a finite number of suspensions. So we want to work in a category, where the suspension functor is inverted, i.e. is an equivalence of categories. Modern accounts of homotopy theory use model structures and there are several constructions of model categories giving the right stable homotopy category. An early one was \cite{BF:gamma}. However, we also want to get the right symmetric monoidal product -- the smash product -- on the model level. Only much later, see \cite{EKMM} and \cite{HSS:sym}, were model categories found, that also support a symmetric monoidal structure inducing the known smash product on the homotopy level. All of them involve a notion of spectrum. 

Spectra are closely connected to linear functors, i.e. generalized homology theories in the sense of Eilenberg and Steenrod. Every spectrum represents a homology theory, and every homology theory is represented by a spectrum by Brown's representability theorem. It is very useful to observe here, that spectra themselves in all the different variants are simplicially enriched functors on one or another simplicially enriched category with values in simplicial sets. 

Manos Lydakis in \cite{Lydakis} put a model structure on the category of functors from pointed finite simplicial sets to pointed simplicial sets, such that the homotopy types correspond exactly to linear functors i.e. reduced generalized homology theories. Moreover he showed, that his category supports a symmetric monoidal product with a simplicial symmetric monoidal Quillen equivalence to the models for spectra. This was motivated by the Goodwillie tower of a homotopy functor \cite{Goo:calc3}, whose first stage is exactly described by Lydakis' construction. He also showed, that the homotopy category associated to this functor category is symmetrically monoidally equivalent to the classical stable homotopy category.

Until now we have only talked about stabilization with respect to suspension: that is smashing with $S^1$. But stabilization with respect to another object has gained interest. In motivic homotopy theory one considers the category of simplicial (pre-)sheaves over smooth finite dimensional schemes over a base scheme. This category is a merger of simplicial sets and smooth schemes. Then one stabilizes it with respect to $\mathbbm{P}^1$, which is the smash product of the simplicial circle $S^1$ and the Tate circle $\mathbbm{A}^1-\{0\}$. The applications of this idea by Voevodsky and many others are now famous. References include \cite{morel-voev:mot-hom} and \cite{jar:motsymspec}.
So it is legitimate to ask for a general theory of stabilization. 

We will address the following setting: Given a symmetric monoidal model category \mc{V}, let us consider the category of \mc{V}-model categories, whose morphisms are given by \mc{V}-Quillen adjunctions. 
Now take a \mc{V}-model category \mc{M} and a \mc{V}-small cofibrant object $L$. 
Tensoring with $L$ induces a left \mc{V}-Quillen functor on \mc{M}, whose right adjoint will commute with filtered homotopy colimits. 
Under certain technical conditions we can associate to a \mc{V}-category \mc{M} another \mc{V}-category, on which $L$ acts as a Quillen equivalence. This will be called an $L$-stabilization of \mc{M}. The special case $\mc{M}=\mc{V}=\mc{S}_*$, pointed simplicial sets, and $L=S^1$, the simplicial circle, brings us back to the initial situation of stable homotopy theory.

From the spectrum point of view this has been investigated by Hovey in \cite{Hovey:general-sym-spec}. We develop here Lydakis' point of view. Under certain conditions they turn out to be equivalent.

To carry this out we study functor categories  from certain full subcategories \mc{U} of \mc{V} to the given \mc{V}-model category \mc{M}.
We first consider in section \ref{section:projective model structure} the projective model structure on the category of small functors from \mc{U} to \mc{M}. Small functors were introduced to homotopy theory in \cite{Chorny-Dwyer:small} and to Goodwillie calculus in \cite{BCR:calc}. They are exactly the right tool to deal with set-theoretic problems in functor categories with non-small source category. We use work of Day and Lack \cite{Day-Lack:limits}, that supplies the completeness of our functor categories with non-small source.

In section \ref{section:homotopy model} we localize the projective model structure, so that the objectwise fibrant homotopy functors become the fibrant objects. A homotopy functor is a functor, that preserves weak equivalences. In this localization step we use the assumption, that every object in \mc{U} is cofibrant. 
If \mc{U} is small, this condition can probably be relaxed. But since we can make good use of it in the later development, we just keep it.
It is a technical weakness, that has to be dealt with in future work. 

Finally in section \ref{section:L-stable functors} we localize further, so that the objectwise fibrant $L$-stable homotopy functors become the fibrant objects.
What are $L$-stable functors? By a general fact small functors are \mc{V}-functors, so they come equipped with a natural map
    $$ \dgrm{X}(\frei)\otimes L\to\dgrm{X}(\frei\otimes L).$$
Let us call the left \mc{V}-Quillen functor obtained by tensoring with $L$ simply $L$ and its right adjoint $R$.
A homotopy functor $\dgrm{X}\co\mc{V}\to\mc{M}$ is $L$-stable if the adjoint map
    $$ \dgrm{X}\to R\circ\dgrm{X}\circ L=\mc{V}(L,\dgrm{X}(\frei\otimes L)) $$
is an objectwise weak equivalence. Observe, that by \cite[Thm 1.8]{Goo:calc3} $S^1$-stable functors are exactly those with a Mayer-Vietoris sequence. In the $L$-stable structure every small functor is weakly equivalent to an $L$-stable one. For $L=S^1$ this is Lydakis' approach to the stable homotopy category. 

Lydakis has shown, how to provide a tensor product on the functor category, that induces the right symmetric monoidal structure on homotopy level. 
This is actually a special case of a more general procedure devised by Day in \cite{day:closed}.
We prove compatibility results with our various model structures on the functor categories in section \ref{section:sym. mon. str.}. 
We obtain in \ref{V^V symmetric monoidal}, that the category $\mc{V}^{\mc{U}}$ is a closed symmetric monoidal model category. The category $\mc{M}^{\mc{U}}$ is a $\mc{V}^{\mc{U}}$-category. 

For small source category \mc{U} we can study functoriality with respect to \mc{M}. We prove, that there is a 2-functor from \mc{V}-model categories with certain properties to itself with values in $L$-stable categories, i.e. those on which $L$ acts as a Quillen equivalence. 
To study functoriality of our $L$-stabilization with respect to the object $L$ is less straightforward as in \cite{Hovey:general-sym-spec}. 
With an auxiliary $L$-$L'$-stable model structure we show in \ref{L-functoriality2}, that the Quillen equivalence type of $\mc{M}^{\mc{U}}$ with the $L$-stable model structure depends only on the weak homotopy type of $L$. 

Finally we compare our $L$-stabilizations with symmetric $L$-spectra and Bousfield-Friedlander $L$-spectra constructed by Hovey in \cite{Hovey:general-sym-spec}. 
Using a certain subcategory $\vsph_L=\mc{U}$ resembling the spheres as a source category, we can prove, that there is a zig-zag of \mc{V}-Quillen equivalences from our $L$-stable model structure on $\mc{M}^{\vsph_L}$ to Hovey's symmetric $L$-spectra on \mc{M}. This is only suboptimal, since there is a canonical candidate for a direct Quillen equivalence. By comparing $L$-stable func\-tors to Bousfield-Friedlander $L$-spectra we can get sufficient conditions, when all three models are Quillen equivalent. This is discussed in sections \ref{section:L-spectra} and \ref{section:examples}.

This is not the first time a generalization of Lydakis' work was attempted.
For certain small \mc{U} and $\mc{M}=\mc{V}$ and general $L$ Dundas, R\"ondigs and \O stv\ae r in \cite{DRO:enriched} described $L$-stabilizations. 
Following them we can also prove the monoid axiom in certain settings \ref{monoid axiom}.
This article generalizes further and greatly simplifies their work on $L$-stable functors.
\\

\noindent
{\bf Acknowledgments:} This work is a side track of my joint project with Boris Chorny and Oliver R\"ondigs to define Goodwillie calculus of homotopy functors for general model categories. I am especially grateful to both of them to point out to me the right generators for the acyclic cofibrations in the various model categories. I would also like to thank Gerald Gaudens and Bj\o rn Dundas for helpful conversations. Finally I would like to thank Steve Lack for keeping me updated on his joint work on \cite{Day-Lack:limits}.

This article was written partly at the University of Western Ontario and partly during my guest stay at the Fields Institute, Toronto.

\section{Closed symmetric monoidal model categories}
\label{section:monoidal model}

We will first describe compatibility conditions between a model structure and a closed symmetric monoidal structure. The basic references for enriched category theory are \cite{Bor:2} and \cite{Kelly}. A good reference on enriched model categories is \cite{Hov:model}.

Let \mc{V} be a closed symmetric monoidal category. We will denote the monoidal functor by $\otimes$ and the unit by $S$. A \mc{V}-category \mc{M} is a category, which is enriched, tensored and cotensored over \mc{V}, where these structures satisfy some adjunction relations. The enrichment given by the \mc{V}-object of maps in \mc{M} will be denoted by $\mc{M}(A,B)$ for objects $A,B$ in \mc{M}. The cotensor will be denoted by $X^A$ for $X$ in \mc{M} and $A$ in \mc{V}. Of course, for $\mc{M}=\mc{V}$ we have $\mc{V}(A,X)=X^A$.

\begin{definition}
Assume, that \mc{V} and \mc{M} are both complete and cocomplete. 
Let $i\co A\to B$ be a map in \mc{V} and $j\co C\to D$ be map in \mc{M}. The {\it pushout product} is the map
    $$ i\,\square\, j\co(A\otimes D)\sqcup_{(A\otimes C)}(B\otimes C)\to B\otimes D $$
in \mc{M} obtained from the universal property of the pushout. There is an adjoint construction 
    $$p\boxtimes i\co X^B\to Y^B\times_{Y^A}X^A $$
for a map $p\co X\to Y$ in \mc{M}.
\end{definition}

\begin{definition}
Let \mc{V} be a closed symmetric monoidal category equipped with a model structure. Then \mc{V} is a {\it closed symmetric monoidal model category}, if the two structures are compatible in the following sense: 
\begin{punkt} 
   \item For two cofibrations $i$ and $j$ in \mc{V} their pushout product $i\,\square\, j$ is a cofibration, which is acyclic if either $i$ or $j$ is acyclic.
   \item For a cofibrant replacement $QS\to S$ of the unit $S$ the induced map $QS\otimes X\to S\otimes X\cong X$ is a weak equivalence for every object $X$ in \mc{V}.
\end{punkt} 
\end{definition}

The dual formulation of (i) is that for a fibration $p$ the map $p\boxtimes i$ is a fibration, which is acyclic if either $i$ or $p$ is acyclic. 
Later on we will assume, that all objects in \mc{V} are cofibrant. In particular the unit $S$ will be cofibrant turning condition (ii) in the previous definition redundant.  Similar remarks apply for the next definition.

\begin{definition} \label{V-model category}
Let \mc{V} be a closed symmetric monoidal model category. Let \mc{M} \mc{V}-category equipped with a model structure. \mc{M} is a {\it \mc{V}-model category}, if the two structures are compatible in the following sense: 
\begin{punkt} 
   \item For a cofibrations $i$ in \mc{V} and a cofibration $j$ in \mc{M} their pushout product $i\,\square\, j$ is a cofibration, which is acyclic if either $i$ or $j$ is acyclic.
   \item For a cofibrant replacement $QS\to S$ of the unit $S$ the induced map $QS\otimes X\to S\otimes X\cong X$ is a weak equivalence for every object $X$ in \mc{M}.
\end{punkt} 
\end{definition}

For the notion of cofibrantly generated model category see {\it references}.
\begin{lemma} \label{smash compatible and generators}
Let $I$ and $J$ be sets of generating $($acyclic$)$ cofibrations for \mc{V}. Suppose $I\,\square\, I\subset {\rm cof}$ and $I\,\square\, J\subset {\rm acof}$ and that the unit $S$ is cofibrant. Then \mc{V} is a closed symmetric monoidal model category.
\end{lemma}

The proof is a straightforward exercise in adjunctions and lifting properties. The precise form of the next definition can be read in \cite[Def. 6.2.3.]{Bor:2} or \cite[p. 9]{Kelly}.
\begin{definition}
Let \mc{M} and \mc{N} be two \mc{V}-categories.
A {\it \mc{V}-functor} $\dgrm{X}\co\mc{M}\to\mc{N}$ is a function ${\rm Ob}(\dgrm{X})\co{\rm Ob}(\mc{M})\to{\rm Ob}(\mc{N})$ together with maps 
    $$ \mc{M}(A,B)\to\mc{N}(\dgrm{X}(A),\dgrm{X}(B)) $$
for all objects $A,B$ in \mc{M}, such that certain diagrams commute. They assert, that \dgrm{X} behaves well with respect to composition and the identity.
\end{definition}

For a \mc{V}-functor $\dgrm{X}\co\mc{V}\to\mc{M}$ and objects $K$ and $L$ in \mc{V} we have canonical maps
\begin{align}\begin{split}
    S\to\mc{V}(K\otimes L,K\otimes L)\cong\mc{V}(K,K\otimes L)^L &\to\mc{V}(\dgrm{X}(K),\dgrm{X}(K\otimes L))^L \\
               &\cong\mc{V}(\dgrm{X}(K)\otimes L,\dgrm{X}(K\otimes L)). 
\end{split}\end{align}
The resulting map 
\begin{equation}\label{assembly1}
    \dgrm{X}(\frei)\otimes L\to\dgrm{X}(\frei\otimes L)
\end{equation}
will be called {\it assembly map}.
The fact, that \dgrm{X} is a \mc{V}-functor can also be expressed in terms of these maps. We will need these maps to define $L$-stability of functors. 

\begin{remark}\label{V-mod}
There is a notion of {\it \mc{V}-natural transformation}, we refer the reader to \cite[p.9]{Kelly}.
For the definition of {\it \mc{V}-Quillen adjunction} we refer to \cite[Def. 4.2.18.]{Hov:model}. It is explained there, that \mc{V}-model categories together with \mc{V}-Quillen adjunction and \mc{V}-natural transformations form a $2$-category, which we will denote by \mc{V}-mod.
We will usually refer to a morphism in this category by its left adjoint. Later on \ref{ulvmod} we will define a full subcategory $\ulvmod$ of \mc{V}-mod, that will be of primary interest.
\end{remark}

We want to generalize the fact from simplicial model categories, that weak equivalences between fibrant and cofibrant objects can be detected with the use of a simplicial interval $\Delta^1$. We will first introduce left \mc{V}-homotopy and then \mc{V}-homotopy equivalence, which actually should be called left \mc{V}-homotopy equivalence.

\begin{definition}\label{cyl(S)}
For the cofibrant unit $S$ of \mc{V} let $\cyl(S)$ denote an object obtained by factoring the fold map $S\sqcup S\to S$ into a cofibration $i\co S\sqcup S\to\cyl(S)$ followed by a weak equivalence $p\co\cyl(S)\to S$. 
We have two inclusions $S\to S\sqcup S$, which we can compose with the map $i$ to obtain two cofibrations $i_0$ and $i_1\co S\to\cyl(S)$. These are section of the map $p$ and hence weak equivalences.
We have the following diagram:
\diagr{ S\sqcup S \ar[r] & CS \ar[r]^-p & S \ar@/^10pt/[l]^-{i_0}\ar@/_12pt/[l]_-{i_1}  }  
We emphasize, that for the moment there is a whole class of choices for $\cyl(S)$.
\end{definition}

\begin{lemma} \label{cylinder}
Let \mc{V} be a closed symmetric monoidal model category with cofibrant unit $S$ and let \mc{M} be a \mc{V}-model category. Let $CS$ be a cylinder object over $S$ and let $X$ be a cofibrant object in \mc{M}. Then $CS\otimes X$ becomes a cylinder object for $X$.
\end{lemma}

\begin{proof}
Since $X$ is cofibrant, tensoring with $X$ preserves cofibrations and acyclic cofibrations. So, if we tensor the diagram with $X$, we obtain a cofibration $X\sqcup X\cong (S\sqcup S)\otimes X\to CS\otimes X$ and a map $CS\otimes X\to S\otimes X\cong X$, that has the acyclic cofibrations $i_0\otimes\id_X$ and $i_1\otimes\id_X\co X\cong S\otimes X\to CS\otimes X$ as sections. By 2-out-of-3 $CS\otimes X\to X$ is a weak equivalence. Hence $CS\otimes X$ is a cylinder object for $X$.
\end{proof}

\begin{definition}
Two maps $f$ and $g\co X\to Y$ in \mc{M} are {\it left \mc{V}-homotopic} if there exists a map $H\co CS\otimes  X\to Y$, called a {\it left \mc{V}-homotopy}, such that the following diagram commutes
\diagr{ X \ar[d]\ar[dr]^f & \\
        CS\otimes X \ar[r]^H & Y \\
        X \ar[u]\ar[ur]_g &  } 
where $CS$ is some cylinder object for $S$.
\end{definition}

For example if \mc{S} is the category of simplicial sets, then for cofibrant $X$ \mc{S}-homotopy is just simplicial homotopy.

\begin{lemma} 
If $X$ is cofibrant, then the \mc{V}-homotopy relation on $\mc{M}(X,Y)$ is an equivalence relation.
\end{lemma}

\begin{proof}
We first prove the special case $X=\emptyset$.
The \mc{V}-homotopy relation is reflexive, since $S$ is a cylinder object for $S$. It is symmetric by switching the inclusions. It is transitive, because the following pushout diagram again forms a cylinder object for $S$:
\diagr{ & & S \ar[d] \\
        & S \ar[r]\ar[d]\ar@{}[dr]|->>>{\pushout} & CS \ar[d] \\
        S \ar[r] & CS' \ar[r] & CS''  }
The general result now follows by tensoring with $X$. 
\end{proof}

\begin{lemma} \label{left=V} 
If $X$ is cofibrant, then the \mc{V}-homotopy relation on $\mc{M}(X,Y)$ coincides with the left homotopy relation.
\end{lemma}

\begin{proof}
This follows directly from lemma \ref{cylinder}.
\end{proof}

\begin{remark}\label{Whitehead}
The previous lemma \ref{left=V} implies the \mc{V}-enriched Whitehead lemma: a map between fibrant and cofibrant objects is a weak equivalence if and only if it is \mc{V}-homotopy equivalence.
Observe also, that \mc{V}-enriched functors preserve \mc{V}-homotopy and in particular \mc{V}-homotopy equivalence. 
\end{remark}

For later use we describe here the \mc{V}-mapping cylinder as an analog of the simplicial mapping cylinder.
\begin{definition}\label{cyl(f)}
Let $f\co A\to B$ be a map in a \mc{V}-model category \mc{M} between cofibrant objects. Then the {\it \mc{V}-mapping cylinder} of $f$ is defined by the following pushout square:
\diagr{ A \ar[r]^-{i_0}\ar[d]^f\ar@{}[dr]|->>>{\pushout} & A\otimes \cyl(S) \ar[d] \\ B \ar[r] & \cyl_{\mc{V}}(f) }
Here $\cyl(S)$ is a fixed cylinder object of the cofibrant unit $S$ of \mc{V} and $i_0\co A\to A\otimes CS$ is induced by the inclusion $i_0$ from \ref{cyl(S)}. There are maps $i_1(f)\co A\to\cyl_{\mc{V}}(f)$ induced by the inclusion $i_1$ from \ref{cyl(S)} and $p(f)\co\cyl(f)\to B$. Here $i_1(f)$ is a cofibration and and $p(f)$ is a weak equivalence. We have a commutative diagram:
\diagr{ A \ar[rr]^-f\ar[dr]_{i_1(f)} & & B  \\  & \cyl_{\mc{V}}(f) \ar[ur]_{p(f)}^{\simeq} }
We will usually drop the reference to \mc{V} and simply denote the cylinder by $\cyl(f)$.
\end{definition}



\section{Localization}

This short section is a quick crash course in the localization technique we will be using.
\begin{definition}\label{endofunctor}
Given an endofunctor $F\colon \mc{C}\to \mc{C}$ in a model category \mc{C} equipped with a coaugmentation $\epsilon\co\id\to F$ we call a map $X\to Y$ in \mc{C} an $F$-\emph{equivalence}, if it induces a weak equivalence $FX\to FY$. A map $X\to Y$ is called an {\it $F$-fibration}, if it has the right lifting property with respect to all projective cofibrations, which are also $F$-equivalences. 
\end{definition}

The following theorem is proved in \cite[A.7]{BF:gamma} and \cite[9.3]{Bou:telescopic}.
\begin{theorem}[Bousfield-Friedlander]\label{thm:bousfield-machine}
Suppose $\epsilon \colon \id\to F$ is a coaugmented endofunctor of a right proper model category $\mc{C}$ satisfying the following axioms:
\begin{description}
  \item[(A.4)] The functor $F$ preserves weak equivalences.
  \item[(A.5)] The maps $\epsilon_{F(A)}, F\epsilon_A \colon F(A) \rightrightarrows FF(A)$ are weak equivalences
    for any object $A\in \mc{C}$.
  \item[(A.6)] Consider a pullback diagram 
    \diagr{ {W} \ar[r] \ar[d] & {Y} \ar[d]_p \\
            {X} \ar[r]^{f} & {Z}}
where $p$ is an $F$-fibration and $f$ is an $F$-equivalence. Then ${W}\to Y$ is an $F$-equivalence.
\end{description}
Then the classes of cofibrations, $F$-equivalences and $F$-fibrations form a right proper model structure, which is simplicial or left proper, if the original model structure on $\mc{C}$ is simplicial or left proper. It possesses functorial factorization, if \mc{C} does. 
\end{theorem}

There is the following characterization of $F$-fibrations.
\begin{lemma} \label{F-fibrant}
A  map $p\co X\to Y$ in \mc{C} is an $F$-fibration if and only if it is a fibration, such that the diagram
\diagr{ X \ar[r]\ar[d]_p & FX \ar[d]^{Fp} \\ Y \ar[r] & FY} 
is a homotopy pullback square in the underlying model structure.
\end{lemma}

\section{Small functors}

We want to study categories of functors from full subcategories of \mc{V} to \mc{V}-model categories. In order to have the most flexibility we choose to work with small functors, since we do not want to restrict ourselves only to small source categories.
\begin{definition}\label{small}
A \mc{V}-functor $\dgrm{X}\co\mc{N}\to\mc{M}$ between two \mc{V}-categories $\mc{N}$ and $\mc{M}$ is {\it small} if it is the \mc{V}-left Kan extension from its restriction to some small full subcategory \mc{T} of $\mc{N}$. We call \mc{T} the defining subcategory of \dgrm{X}.
Let $\mc{M}^{\mc{N}}$  denote the category of small functors from $\mc{N}$ to $\mc{M}$. 
\end{definition}

\begin{remark}\label{LKan}
\begin{punkt}
   \item
Small functors are \mc{V}-functors by our definition.
   \item \label{LKanii}
Let $i\co\mc{U}_1\inj\mc{U}_2$ be an inclusion of full subcategories of \mc{V}, then the restriction functor
    $$ i^*\co\mc{M}^{\mc{U}_2}\to\mc{M}^{\mc{U}_1} $$
has a left adjoint 
    $$ \LKan_i\co\mc{M}^{\mc{U}_1}\to\mc{M}^{\mc{U}_2} $$
given by \mc{V}-left Kan extension and the category $\mc{M}^{\mc{U}_1}$ becomes a retract (up to natural equivalence) of the category $\mc{M}^{\mc{U}_2}$. In this way we will view the category $\mc{M}^{\mc{U}}$ as full subcategory of $\mc{M}^{\mc{V}}$. So we can always assume, that our functors are actually defined on \mc{V}. 
\end{punkt}
\end{remark}

\begin{definition}
For an object $V$ in \mc{V} its {\it covariant \mc{V}-representable functor} $\mc{V}\to\mc{V}$ will be denoted by
    $$ R^V(\frei)=\mc{V}(V,\frei).$$
This functor is obviously small with defining subcategory $\{V\}$.
\end{definition}

\begin{remark} \label{properties of smallness}
\begin{punkt}
   \item
If the unit $S$ of \mc{V} is cofibrant, all representables
   \item
It is standard \cite[Prop. 4.83]{Kelly}, that a functor $\dgrm{X}\co\mc{V}\to\mc{M}$ is small if and only if it is an enriched colimit of representables, i.e. there exists a small full subcategory \mc{K} of \mc{V} such that 
    $$ \dgrm{X} \cong \int^{K\in\mc{K}}R^K\otimes\dgrm{X}K. $$ 
Using this and \cite[Prop. 8.3]{Day-Lack:limits}, that if $\dgrm{W}\co\mc{V}\to\mc{V}$ and $\dgrm{Y}\co\mc{M}\to\mc{M}$ are small, then $\dgrm{X}\circ\dgrm{W}$ and $\dgrm{Y}\circ\dgrm{X}$ are small, too. 
   \item
For any full subcategory \mc{U} of \mc{V} the category $\mc{M}^\mc{U}$ is cocomplete by definition. We also want it to be complete. This holds obviously, when \mc{U} is small.
But the remarkable result \cite[Theorem 8.5]{Day-Lack:limits} of Day and Lack shows, that this also holds, when \mc{U} is cocomplete. Limits and colimits are both computed objectwise.
\end{punkt}
\end{remark}

\begin{definition}\label{V + small}
Let \dgrm{X} and \dgrm{Y} be functors in $\mc{M}^{\mc{U}}$ and let $V$ be an object of \mc{V}. We define a {\it tensor} $\otimes\co\mc{M}^{\mc{U}}\times\mc{V}\to\mc{M}^{\mc{U}}$ by
    $$ (\dgrm{X}\otimes V)(K)=\dgrm{X}(K)\otimes V.$$
The {\it enrichment} of $\mc{M}^{\mc{U}}$ is given by the \mc{V}-object of natural transformations of \mc{V}-functors
    $$ \Vnat(\dgrm{X},\dgrm{Y})=\int_{K\in\mc{K}}\mc{V}(\dgrm{X}(K),\dgrm{Y}(K)), $$
where the end is taken over a small full subcategory \mc{K} of \mc{U} defining \dgrm{X}. 
The {\it cotensor} $\mc{V}^{\rm op}\times\mc{M}^{\mc{U}}\to\mc{M}^{\mc{U}}$ is given by
    $$ (\dgrm{X}^V)(K)=\dgrm{X}(K)^V=\mc{V}(V,\dgrm{X}(K)), $$
where the right hand side is the cotensor of the underlying \mc{V}-category \mc{M}.
\end{definition}

Consistency would demand denoting the \mc{V}-enrichment of $\mc{M}^{\mc{U}}$ by $\mc{M}^{\mc{U}}(\dgrm{X},\dgrm{Y})$, but we save this notation for the enrichment of $\mc{M}^{\mc{U}}$ over itself \ref{M^V over itself}.

It is easy to check, that there are canonical isomorphisms
    $$ \Vnat(\dgrm{X},\dgrm{Y}^K)\cong\Vnat(\dgrm{X}\otimes K,\dgrm{Y})\cong\mc{V}(K,\Vnat(\dgrm{X},\dgrm{Y})),$$
so that $\mc{M}^{\mc{U}}$ is indeed a \mc{V}-category.

\section{Fixing notation and first assumptions}
\label{section:fixing}

We reserve the letter \mc{V} for a locally presentable cofibrantly generated symmetric monoidal model category. 
We denote the monoidal product by $\otimes$ and the unit by $S$, which has to be cofibrant.
Let $I_{\mc{V}}$ and $J_{\mc{V}}$ be sets of generating cofibrations and acyclic cofibration.
We additionally assume that \mc{V} possesses functorial fibrant replacements, i.e. for each object $K$ there exists a functorial weak equivalence $K\to K^\fib$ into a fibrant object $K^\fib$. 
We choose one such replacement functor $\fibfun$ and we assume, that it is small, see definition \ref{small}. 

\begin{definition}
An object $L$ in the symmetric monoidal category \mc{V} is called {\it \mc{V}-small}, if the \mc{V}-enriched representable $R^L$ commutes with sequential \mc{V}-colimits, i.e. if the canonical map
    $$ \colim_{n\in\mathbbm{N}}R^L(K_n)=\colim_{n\in\mathbbm{N}}\mc{V}(L,K_n)\to\mc{V}(L,\colim_{n\in\mathbbm{N}}K_n)=R^L(\colim_{n\in\mathbbm{N}}K_n)$$
is an isomorphism for every sequential system $(K_n)_{n\in\mathbbm{N}}$.
\end{definition}

The letter $L$ will always stand for a \mc{V}-small cofibrant object $L$ in \mc{V}. If the occasion arises, we will use $L'$ for another such object with the same properties.

We reserve the letter \mc{U} for a full subcategory of \mc{V}, that is either small or cocomplete. We will gradually add more assumptions on \mc{U}, as we go along, see \ref{Uho}, \ref{U1}, \ref{U2} and \ref{U3}. 

We reserve the letter \mc{M} for a right proper locally presentable cofibrantly generated \mc{V}-model category. Let $I_{\mc{M}}$ and $J_{\mc{M}}$ be sets of generating cofibrations and acyclic cofibration.
Again we assume that \mc{M} possesses functorial small fibrant replacements and we choose such a replacement functor $\fibfun$. 

\begin{definition}\label{ulvmod}
We denote by $\ulvmod$ the full subcategory of \mc{V}-mod \ref{V-mod} of right proper locally presentable cofibrantly generated \mc{V}-model categories possessing a functorial small fibrant replacement. 
\end{definition}

Observe, that $L$ induces left Quillen endofunctors of \mc{M} and \mc{V} 
    $$ K\mapsto L(K)=K\otimes L $$
given by tensoring with $L$  and denoted by the same letter. 
The right adjoint of $L$ will be denoted by 
    $$ R(K)=K^L. $$
It commutes with filtered colimits and homotopy colimits, since $L$ is small.

\begin{remark}\label{alpha-commute}
Let $\alpha$ be a regular cardinal.
We remind the reader, that in an $\alpha$-presentable category $\alpha$-limits commute with $\alpha$-filtered colimits \cite[Cor. 5.2.8]{Bor:2}. In particular, finite homotopy limits commute with filtered homotopy colimits.
\end{remark}

\section{The projective model structure}
\label{section:projective model structure}

Following \cite{Chorny-Dwyer:small} we can equip $\mc{M}^{\mc{U}}$ with the projective model structure: weak equivalences and fibrations are given objectwise. Fibrations are detected by the class $R^U\otimes J_{\mc{M}}$, acyclic fibrations are detected by the class $R^U\otimes I_{\mc{M}}$, where $U$ runs through all objects of \mc{U}. 
Here $I_\mc{M}$ denotes a generating set of cofibrations for \mc{M}, and $J_{\mc{M}}$ denotes a set of generating acyclic cofibrations.
Both classes are locally small, i.e. they satisfy the co-solution set condition, and all the source and target objects are small in the sense, that mapping out of them commutes with filtered colimits. Therefore we can use the generalized small object argument from \cite{Chorny:general} in the same way as in \cite{Chorny-Dwyer:small} to prove the existence of the projective model structure. 

\begin{theorem}\label{V + projective}
The category $\mc{M}^{\mc{U}}$ equipped with the projective model structure is class-cofibrantly generated \mc{V}-model category. It is right or left proper, if \mc{M} is so. If \mc{U} is a small full subcategory of \mc{V}, then $\mc{M}^{\mc{U}}$ is cofibrantly generated and hence possesses functorial factorization. 
\end{theorem}

\begin{proof}
We have to check, that $\mc{M}^{\mc{U}}$ is a \mc{V}-model category. Let $i\co A\to B$ be a cofibration in \mc{V} and let $p\co\dgrm{X}\to\dgrm{Y}$ be a fibration in $\mc{M}^{\mc{U}}$. Then the induced map 
    $$ \dgrm{X}^B\to\dgrm{Y}^B\times_{\dgrm{Y}^A}\dgrm{X}^A $$
is an objectwise fibration, since cotensors are defined objectwise. It is an acyclic fibration if either $i$ or $p$ are acyclic, for the same reason.

If \mc{M} is left proper, then $\mc{M}^{\mc{U}}$ is left proper, because  weak equivalences and cofibrations are in particular objectwise and pushouts are computed objectwise. The same applies to right properness.

If \mc{U} is a small full subcategory, the generating classes of cofibrations and acyclic cofibrations are in fact sets, and so factorization is functorial.
\end{proof}

In general we do not expect to have functorial factorization. The reason is, that we do not know, how to choose the co-solution sets for the generalized small object argument in a functorial way. This leaves all further localizations of the projective model structure on $\mc{M}^{\mc{U}}$ to be without functorial factorization, unless the source category is small. But we think, this is a small price to pay in order to handle more general source categories than just small ones. 

\begin{remark}\label{rep. cof}
If the unit $S$ of \mc{V} is cofibrant, all representable functors $R^K$ are projectively cofibrant.
\end{remark}

\section{The homotopy model structure}
\label{section:homotopy model}
We will now outline, how to localize the projective model structure on $\mc{M}^{\mc{U}}$ to obtain the homotopy model structure, where the fibrant objects are exactly the projectively fibrant homotopy functors. 

\begin{axiom}\label{Uho}
To construct the homotopy model structure on $\mc{M}^{\mc{U}}$ we have to introduce further assumption on \mc{U}:
\begin{enumerate}
    \item
All objects in \mc{U} are cofibrant.
    \item
It is closed under the fibrant approximation functor $\fibfun\co\mc{U}\to\mc{U}$.  
\end{enumerate}
The last condition means, that for every object $K$ there exists a functorial weak equivalence into a fibrant object $K^\fib$, which is still in \mc{U}.
\end{axiom}

In the model category \mc{V} we have a Whitehead lemma \ref{Whitehead}: a map between fibrant and cofibrant objects is a weak equivalence if and only if it is a \mc{V}-homotopy equivalence. If we have a weak equivalence $w$ between fibrant objects in \mc{U}, it is already a \mc{V}-homotopy equivalence, since all objects in \mc{U} are assumed to be cofibrant. 
Now let \dgrm{X} be in $\mc{M}^{\mc{U}}$.
The weak equivalence $w$ is mapped to a \mc{V}-homotopy equivalence $\dgrm{X}(w)$, because small functors are \mc{V}-functors.
Altogether, if we precompose an arbitrary small functor \dgrm{X} with a small fibrant replacement functor $\fibfun$, the resulting functor $\dgrm{X}\fibfun$ will be small as explained in remark \ref{properties of smallness} and it will be a homotopy functor.
To obtain the homotopy model structure we just observe, that the functor $\dgrm{X}\mapsto\dgrm{X}\fibfun$ satisfies the axioms of \ref{thm:bousfield-machine}, as was proved in \cite[Prop. 3.3]{BCR:calc}. 

In the homotopy model structure on $\mc{M}^{\mc{U}}$ a map $\dgrm{X}\to\dgrm{Y}$ is 
\begin{enumerate}
   \item
a cofibration if and only if it is a projective cofibration, 
   \item
a weak equivalences if and only if the induced map
     $$ \dgrm{X}\fibfun\to\dgrm{Y}\fibfun $$
is an objectwise weak equivalence.
   \item
a fibration if and only if the square
\diagr{ \dgrm{X} \ar[r]\ar[d] & \dgrm{X}\fibfun \ar[d] \\ \dgrm{Y} \ar[r] & \dgrm{Y}\fibfun }
is an objectwise homotopy pullback square. Compare \ref{F-fibrant}. 
\end{enumerate}
Fibrant objects are exactly the objectwise fibrant homotopy functors.

\begin{definition}\label{h}
The functor $\fibfun$ is an endofunctor of \mc{U}, that we assume to be small. Then the following functor denoted by $(\frei)^h$
    $$\dgrm{X}\mapsto\fibfun\circ\dgrm{X}\circ\fibfun=\dgrm{X}^h$$ 
is a fibrant replacement functor in the homotopy model structure on $\mc{M}^{\mc{U}}$. 
\end{definition}

\begin{theorem}\label{V + homotopy}
The category $\mc{M}^{\mc{U}}$ equipped with the homotopy model structure is right proper \mc{V}-model category. It is left proper if \mc{M} is so. 
\end{theorem}

\begin{proof} 
We have to show, that for a homotopy fibration $p\co\dgrm{X}\to\dgrm{Y}$ in $\mc{M}^{\mc{U}}$ and a cofibration $i\co K_1\to K_2$ in \mc{V} the induced map
    $$ p\boxtimes i\co \dgrm{X}^{K_2}\to\dgrm{Y}^{K_2}\times_{\dgrm{Y}^{K_1}}\dgrm{X}^{K_1} $$
is a homotopy fibration. But this follows from the characterization of homotopy fibrations in \ref{F-fibrant} and the fact, that there is a natural isomorphism
   $$ (\dgrm{Y}^{K_2}\times_{\dgrm{Y}^{K_1}}\dgrm{X}^{K_1})\fibfun\cong\dgrm{Y}^{K_2}\fibfun\times_{\dgrm{Y}^{K_1}\fibfun}\dgrm{X}^{K_1}\fibfun. $$
The assertion about left properness follows from theorem \ref{thm:bousfield-machine}.
\end{proof}

The following fact is merely a tautology, but it is a useful different point of view.
\begin{lemma}\label{fib + we}
A map is weak equivalence in the homotopy model structure if and only if its restriction to the category of fibrant objects is an objectwise weak equivalence.
\end{lemma}

\begin{question}
The assumption, that all objects in \mc{U} be cofibrant, is unfortunate. Is there a better way to construct the homotopy model structure without this assumption, in particular if \mc{U} is not small?
This needs further investigation. 
\end{question}

To proceed we will now find useful generators for the acyclic cofibrations in the homotopy model structure. 

\begin{definition}\label{J^ho}
Let $f\co A\to B$ be an arbitrary acyclic cofibration in \mc{U} and consider the induced map $f^*\co R^B\to R^A$ of representable functors. Now factor $f^*$ into a projective cofibration $j\co R^B\to\cyl(f^*)$, which in the terminology of \ref{cyl(f)} is given by $j=i_1(f^*)$, followed by the objectwise fibration $p\co\cyl(f^*)\to R^A$. Let us denote the set of all maps of the form $j$, where $f$ runs through the set of all acyclic cofibrations between objects in \mc{U}, by $J'$. Then form the set $J''$ of pushout products $j\,\square\, i$, where $j\in J'$ and $i\in I_{\mc{M}}$.
Finally define the set of maps in $\mc{M}^{\mc{U}}$
\begin{equation*}
    J^{\rm ho}_{\mc{M}^{\mc{U}}}:= J''\cup J_{\mc{M}^{\mc{U}}}^{\rm proj}. 
\end{equation*}
Note that all maps in $J''$ are projective cofibrations by \ref{V + projective}.
\end{definition}

\begin{theorem}\label{generators for ho}
A map in $\mc{M}^{\mc{U}}$ has the right lifting property with respect to the class $J^{\rm ho}_{\mc{M}^{\mc{U}}}$ if and only if it is a fibration in the homotopy model structure.
\end{theorem}

We have not checked the co-solution set condition required for the generalized small object argument of \cite{Chorny:general}. So we do not claim, that the homotopy model structure is class-cofibrantly generated. There is merely a certain class of acyclic cofibrations, that detect fibrations by the lifting property. However if \mc{U} is small, the homotopy model structure is cofibrantly generated, because sources and targets of the generating set are small.

\begin{proof}
We first prove, that a fibration in the homotopy model structure on $\mc{M}^{\mc{U}}$ has the right lifting property with respect to $J^{\rm ho}_{\mc{M}^{\mc{U}}}$. To streamline the argument we use right properness and the fact, that applying $\fibfun$ preserves objectwise fibrations, and reduce to the case of an objectwise $\dgrm{X}\to\dgrm{Y}$ fibration between objectwise fibrant homotopy functors. Obviously such maps have the right lifting property with respect to $J_{\mc{M}^{\mc{U}}}^{\rm proj}$. Given a diagram
\diag{ (\cyl(f^*)\otimes C)\sqcup_{(R^B\otimes C)}(R^B\otimes D) \ar[r]\ar[d]_-{j\,\square\, i} & \dgrm{X} \ar[d] \\ \cyl(f^*)\otimes D \ar[r] & \dgrm{Y} }{lift1}
where $j\in J'$ and $i\in I_{\mc{M}}$. This is adjoint to the diagram:
\diag{ C \ar[r]\ar[d]_{i} & \dgrm{X}^{\cyl(f^*)} \ar[d] \\  D \ar[r] & \dgrm{Y}^{\cyl(f^*)}\times_{\dgrm{Y}(B)}\dgrm{X}(B) }{lift2}
Here $\dgrm{X}^{\cyl(f^*)}$ is the cotensor of $\mc{M}^{\mc{U}}$ over $\mc{V}^{\mc{U}}$ from \ref{M^V over itself}.
Note that the right hand map is an objectwise fibration by the fact \ref{tensor + projective}, that $\mc{M}^{\mc{U}}$ with the projective structure is a $\mc{V}^{\mc{U}}$-model category.
Now one can easily show, that there is a commutative square
\diagr{ \dgrm{X}^{\cyl(f^*)} \ar[d] & \dgrm{X}(A) \ar[l]_-{\simeq}\ar[d]^{\simeq} \\  
       \dgrm{Y}^{\cyl(f^*)}\times_{\dgrm{Y}(B)}\dgrm{X}(B) & \dgrm{Y}(A)\times_{\dgrm{Y}(B)}\dgrm{X}(B) \ar[l]^-{\simeq} }
So a lifting exists in (\ref{lift2}), hence also in the adjoint square (\ref{lift1}).

Conversely let $\dgrm{X}\to\dgrm{Y}$ be a map with the right lifting property with respect to $J^{\rm ho}_{\mc{M}^{\mc{U}}}$. Then it is obviously an objectwise fibration. The previous diagrams show, that there is a weak equivalence
    $$ \dgrm{X}(A)\simeq\dgrm{Y}(A)\times_{\dgrm{Y}(B)}\dgrm{X}(B). $$
We factor the map $K\to K^\fib$ into an acyclic cofibration $K\to K'$ followed by an acyclic fibration $K'\to K^\fib$. This last map is a weak equivalence between fibrant and cofibrant objects, hence it is mapped to a weak equivalence by \dgrm{X} and \dgrm{Y}. We obtain:
    $$ \dgrm{X}(K)\simeq\dgrm{Y}(K)\times_{\dgrm{Y}(K')}\dgrm{X}(K')\simeq\dgrm{Y}(K)\times_{\dgrm{Y}(K^{\fib})}\dgrm{X}(K^{\fib}). $$
So $\dgrm{X}\to\dgrm{Y}$ is a fibration in the homotopy model structure by \ref{F-fibrant}. 
\end{proof}

\section{The $L$-stable model structure}
\label{section:L-stable functors}

Now we will localize this structure further to obtain the $L$-stabilized or $L$-stabilized version. We need further assumptions on \mc{U}. Let us list them here.

\begin{axiom}\label{U1}
In order to gain a little bit more flexibility we will not just consider small functors defined on \mc{V}, but also small functors from a full subcategory $\mc{U}$ of \mc{V}. This full subcategory \mc{U} should satisfy the following assumptions:

\begin{enumerate}
    \item
It is either small or cocomplete.
    \item
All objects in \mc{U} are cofibrant.
    \item
It is closed under the fibrant approximation functor $\fibfun$.
    \item
It contains the unit $S$.
    \item
It is closed under tensoring with $L$ (and $L'$).
\end{enumerate}
\end{axiom}


In order to compare different linearizations we need a slightly more general notion. 

\begin{remark}\label{L-assembly remark}
For a small functor $\dgrm{X}\co\mc{U}\to\mc{M}$ and every object $K$ in \mc{V} there is an assembly map 
\begin{equation}\label{L-assembly map}
    (L\circ\dgrm{X})(K)=\dgrm{X}(K)\otimes L \to\dgrm{X}(K\otimes L)=(\dgrm{X}\circ L)(K), 
\end{equation}
since small functors are \mc{V}-functors, compare (\ref{assembly1}). If $L'\to L$ is a morphism in \mc{V} we have a commutative diagram:
\diag{ \dgrm{X}(K)\otimes L' \ar[r]\ar[d] & \dgrm{X}(K\otimes L') \ar[d] \\
       \dgrm{X}(K)\otimes L \ar[r] & \dgrm{X}(K\otimes L)  }{L-L'-assembly map}
The composed diagonal map is called the $L$-$L'$-assembly map.
\end{remark}

\begin{definition}\label{L-stable def}
A functor $\dgrm{X}\co\mc{U}\to\mc{M}$ is called {\it $L$-$L'$-stable}, if it is a homotopy functor and the adjoint map of (\ref{L-L'-assembly map})
    $$ t_{\dgrm{X},K}\co\dgrm{X}(K)\to (R'\circ\dgrm{X}\circ L)(K)$$
is a weak equivalence for all objects $K$ in \mc{U}. If $L=L'$, (\ref{L-L'-assembly map}) reduces to (\ref{L-assembly map}). In that case such a functor is simply called $L$-stable. This is the case, we are mostly interested in.
\end{definition}

\begin{remark}
For pointed simplicial sets $\mc{V}=\mc{S}_*$ the $S^1$-stable functors are simply the linear functors in the sense of Goodwillie \cite{Goo:calc3}. 
\end{remark}

\begin{definition}\label{reduced}
Let $\ast$ be the terminal object of \mc{V} and by abuse of language also of \mc{M}.
A functor $\dgrm{X}\co\mc{V}\to\mc{M}$ is called {\it reduced}, if $\dgrm{X}(\ast)\simeq\ast$.
\end{definition}

\begin{remark}
Note, that if the category \mc{V} is pointed, in the sense that the initial object is isomorphic to the terminal one $\ast$, then all small functors in $\mc{M}^{\mc{U}}$ are reduced. 
But we point out, that in general \mc{V} does not have to be pointed for the construction of the $L$-stable model structures. 
\end{remark}

Recall, that $(\frei)^h$ is a fibrant replacement functor in the homotopy model structure, see \ref{h}.

\begin{definition}
We define the following functor from $\mc{M}^{\mc{U}}$ to itself:
    $$ P_L^{L'}\dgrm{X}=\hocolim_{n}R'^n\circ\dgrm{X}^h\circ L^n.$$
Here, of course, $R'$ is the right adjoint to $L'$ given by cotensoring with the object $L'$.
For $L=L'$ we will write $P_L$ or, if no confusion can arise, we will drop the reference to $L$ and $L'$ at all.
The functor $P_L^{L'}=P$ is coaugmented. Let $p\co\id\to P_L^{L'}$ denote the coaugmentation. There is a commutative diagram
\diag{ P_L^{L'} \ar[r]\ar[d] & P_L \ar[d] \\
       P_{L'} \ar[r] & P_{L'}^{L} }{L-functoriality1}
of coaugmented functors.
\end{definition}

\begin{remark}\label{about P}
\begin{punkt}
    \item If $\dgrm{X}$ is a homotopy functor, then $P\dgrm{X}$ is a homotopy functor as well. This follows from the assumption, that all objects are cofibrant. The functor $L$ preserves weak equivalences between cofibrant objects. $R'$ and $\dgrm{X}$ preserve weak equivalences, because of the interspersed fibrant replacements.
    \item The functor $P$ commutes with finite homotopy limits and filtered homotopy colimits because of \ref{alpha-commute}.
    \item From the assumption, that $R'$ commutes with filtered homotopy colimits, it follows directly, that we have a natural equivalence $P\circ R'\fibfun\simeq R'\circ P$.
    \item Therefore, for each $\dgrm{X}$ the functor $P\dgrm{X}$ is $L$-stable. 
    \item From \ref{properties of smallness} it follows, that $P\dgrm{X}$ is again small, if $\dgrm{X}$ was. 
\end{punkt}
\end{remark}

\begin{lemma}\label{P}
The functor $P$ satisfies properties {\rm\bf (A.4)}, {\rm\bf (A.5)} and {\rm\bf (A.6)} from \emph{\ref{thm:bousfield-machine}}. 
\end{lemma}

\begin{proof}
The fact, that $P$ preserves projective weak equivalences, follows from the fact, that $R'$ preserves weak equivalences between fibrant objects. This is the reason, why a fibrant replacement functor had to be interposed. This proves {\rm\bf (A.4)}. 

The proof of {\bf (A.5)} is taken from \cite{Goo:calc3}.
It is clear, that the map $pP\dgrm{X}$ is a weak equivalence, because $P\dgrm{X}$ is $L$-$L'$-stable and therefore the map 
    $$R'^n\dgrm{X}^hL^n\to R'^{n+1}\dgrm{X}^hL^{n+1}$$
is a weak equivalence. Next we consider:
\begin{align*}
    P\dgrm{X} \to P(R'\circ\dgrm{X}^h\circ L)\toh{\simeq}R'\circ P\dgrm{X}\circ L
\end{align*}
This composition is a weak equivalence, because $P\dgrm{X}$ is $L$-$L'$-stable. Therefore the first map of the previous composition is a weak equivalence. This easily implies, that $Pp\dgrm{X}$ is a weak equivalence proving {\bf (A.5)}. 

Finally we have to prove {\bf (A.6)}. But this follows easily from the fact, that $P$ commutes with homotopy pullbacks by \ref{about P}(ii) and that \mc{M} is right proper. 
\end{proof}
\begin{remark}
It is worth pointing out, that the functor $P$ satisfies the axioms {\bf (A.4)}, {\bf (A.5)} and {\bf (A.6)} even before localizing to the homotopy model structure. This means, that there exists an $L$-stable model structure for non-homotopy functors. We do not know, how to interpret this.
\end{remark}

Now we can apply theorem \ref{thm:bousfield-machine} once again. We obtain the $L$-$L'$-stable model structure on $\mc{M}^{\mc{U}}$. Cofibrations are still the projective ones. Before we summarize our findings in the next theorem, let us define explicitly the model structure and we prove another characterization of $L$-$L'$-stable fibrations. 

\begin{definition}\label{L-stable structure}
A map $\dgrm{X}\to\dgrm{Y}$ is an \emph{$L$-$L'$-stable equivalence} if and only if $P\dgrm{X}\to P\dgrm{Y}$ is a homotopy weak equivalence or, in fact, a projective weak equivalence, since $P\dgrm{X}$ and $P\dgrm{Y}$ are homotopy functors. \emph{$L$-$L'$-stable fibrations} are homotopy fibrations $\dgrm{X}\to\dgrm{Y}$, such that the diagram
\diagr{ \dgrm{X} \ar[r]\ar[d] & P\dgrm{X} \ar[d] \\ \dgrm{Y} \ar[r] & P\dgrm{Y} }
is a homotopy pullback diagram in the homotopy model structure. 
\end{definition}

From now on we will also assume that $L=L'$ for ease on notation and remark, that the proofs are literally the same for the more general $L$-$L'$-stable model structure.
\begin{lemma}\label{L-stable fibrations}
A map $\dgrm{X}\to\dgrm{Y}$ is an $L$-stable fibration if and only if it is a homotopy fibration such that the following diagram
\diagr{ \dgrm{X} \ar[r]\ar[d] & R\dgrm{X}^h L \ar[d] \\
        \dgrm{Y} \ar[r] & R\dgrm{Y}^h L }
is a homotopy pullback square in the homotopy model structure.
\end{lemma}

\begin{proof}
According to \ref{F-fibrant} an $L$-stable fibration is a map as described in definition \ref{L-stable structure}. We have a diagram
   $$\xy
   \xymatrix"*"@=13pt{ R\dgrm{X}^hL \ar'[d] [dd] \ar[rr] & & RP\dgrm{X}L \ar[dd] \\
                                 & &                \\
                       R\dgrm{Y}^hL \ar'[r] [rr] & &  RP\dgrm{Y}L  }
   \POS(-10,-9)
   \xymatrix@=19pt{ \dgrm{X} \ar[rr]\ar[dd]\ar["*"]  & & P\dgrm{X} \ar[dd]\ar["*"]_{\simeq}  \\
                                 & &                \\
                    \dgrm{Y} \ar[rr]\ar["*"]   & & P\dgrm{Y} \ar["*"]_{\simeq}   }
   \endxy   $$
where the front and the back square are homotopy pullbacks in the homotopy model structure. Since we have weak equivalences on the right hand side due to fact, that $P\dgrm{X}$ and $P\dgrm{Y}$ are $L$-stable, the second square in the following diagram is a homotopy pullback: 
\diagr{ \dgrm{X} \ar[r]\ar[d] & R\dgrm{X}^hL \ar[r]\ar[d] & P\dgrm{X} \ar[d] \\
        \dgrm{Y} \ar[r] & R\dgrm{Y}^hL \ar[r] & P\dgrm{Y} }
The combined square is a homotopy pullback, so is the left hand square.
\end{proof}

\begin{theorem}\label{V + L-stable}
The category $\mc{M}^{\mc{U}}$ equipped with the $L$-stable model structure is a right proper \mc{V}-model category. It is left proper if \mc{M} is so. 
\end{theorem}

\begin{proof}
The existence of the model structure follows from \ref{thm:bousfield-machine} once again using the fact, that the functor $P$ satisfies the necessary axioms by \ref{P}. This also covers the left properness assertion.
It remains to show, that for an $L$-stable fibration $\dgrm{X}\to\dgrm{Y}$  and a cofibration $i\co K_1\to K_2$ in \mc{V} the induced map
    $$ \dgrm{X}^{K_2}\to\dgrm{Y}^{K_2}\times_{\dgrm{Y}^{K_1}}\dgrm{X}^{K_1}$$
is an $L$-stable fibration. But this follows directly from the characterization of $L$-stable fibrations in \ref{L-stable fibrations} and the fact, that there is a natural isomorphism
    $$ R(\dgrm{X}^K)L\cong(R\dgrm{X}L)^K.$$
\end{proof}

General model categorical considerations give the following lemma.  
\begin{lemma}
For each functor $\dgrm{X}$ the functor $P\dgrm{X}$ is $L$-stable. Moreover, if $\dgrm{X}\to\dgrm{Y}$ is a map in $\mc{M}^{\mc{U}}$ to an $L$-stable functor $\dgrm{Y}$, then there exists a zig-zag of maps 
     $$ P\dgrm{X}\stackrel{\simeq}{\twoheadleftarrow}\dgrm{C}\to\dgrm{Y} $$ 
under \dgrm{X}, where the first arrow is an objectwise acyclic fibration.
\end{lemma}
We leave the precise uniqueness statement of the zig-zag to the reader. 

Now we will describe a certain class of $L$-stably acyclic cofibrations, that detect $L$-stable fibrations by the lifting property. As for the homotopy model structure we do not prove the co-solution set condition for this class and we do not claim, that the $L$-stable model structure is class-cofibrantly generated. Again it will be cofibrantly generated as soon as \mc{U} is small.

Let $\mc{V}(L,R^K)$ denote the cotensor from \ref{V + small} of $L$ with the representable functor $R^K=\mc{V}(K,\frei)$ for $K$ in \mc{U}. Adjoint to the identity of $\mc{V}(L,R^K)$ is the map
   $$ \tau_K\co R^{K\otimes L}\otimes L\cong\mc{V}(L,R^K)\otimes L\to R^K .$$
Note that $\mc{V}(\tau_K,\dgrm{X})=t_{\dgrm{X},K}$, where $t_{\dgrm{X},K}$ is defined in \ref{L-stable def}.
So a functor \dgrm{X} in $\mc{M}^{\mc{U}}$ is $L$-stable exactly if the cotensor \ref{M^V over itself} with this map is a weak equivalence for every $K$. Observe, that $\tau_K$ is a map between projectively cofibrant objects, so we can form the \mc{V}-mapping cylinder \ref{cyl(f)}:
\diagr{ R^{K\otimes L}\otimes L\ar[r]^-{j(K)} & \cyl(\tau_K)\ar[r]^-{p(K)} & R^K }
Here $j(K)$ is a projective cofibration and $p(K)$ is an objectwise equivalence. Let $J'''$ be the class of maps of the form $j(K)\,\square i$, with $j(K)$ as above for all $K$ in \mc{U} and where $i$ runs through the class $I_\mc{M}$. Note, that all maps in $J'''$ are projective cofibrations.

\begin{definition}\label{J^lin}
We define a new class of maps by
\begin{equation*}
    J^{L{\rm -stable}}_{\mc{M}^{\mc{U}}}:= J'''\cup J_{\mc{M}^{\mc{U}}}^{\rm ho}. 
\end{equation*}
\end{definition}

\begin{theorem}
A map in $\mc{M}^{\mc{U}}$ has the right lifting property with respect to the class $J^{L{\rm -stable}}_{\mc{M}^{\mc{U}}}$ if and only if it is a fibration in the $L$-stable model structure.
\end{theorem}

The proof is similar to the proof of \ref{generators for ho} and will be omitted. \\

Now we will prove, that the $L$-stabilization really deserves its name: the action of $L$ on $\mc{M}^{\mc{U}}$ induces an equivalence of the associated homotopy categories.

\begin{definition}\label{L-stable}
Let $\mc{N}$ be a \mc{V}-model category. If the functor $L\co\mc{N}\to\mc{N}$ given by tensoring with $L$ is a \mc{V}-left Quillen equivalence, we call \mc{N} {\it $L$-stable}.
\end{definition}

\begin{theorem}
The action of $L$ on $\mc{M}^{\mc{U}}$ given by the tensor from \emph{\ref{V + small}} is the left adjoint of a \mc{V}-Quillen equivalence with right adjoint given by cotensoring with $L$.
The category $\mc{M}^{\mc{U}}$ equipped with the $L$-stable model structure is $L$-stable.
\end{theorem}

Again we will denote tensoring with $L$ just by $L\dgrm{X}=\dgrm{X}\otimes L$ and cotensoring with $R\dgrm{X}=\dgrm{X}^L$.

\begin{proof}
It is obvious, that $(L,R)$ form a Quillen pair.
It remains to show, that $(L,R)$ is a Quillen equivalence from $\mc{M}^{\mc{U}}$ to itself. Let $\dgrm{X}\to R\dgrm{Y}$ be an $L$-stable equivalence, where $\dgrm{X}$ is projectively cofibrant and $\dgrm{Y}$ is $L$-stably fibrant.  Let $\dgrm{X}\inj P'\dgrm{X}\toh{\simeq} P\dgrm{X}$ be a factorization of the coaugmentation $p\dgrm{X}$ into a projective cofibration followed by a projective weak equivalence. Since $\dgrm{Y}$ is $L$-stable, it follows easily, that $R\dgrm{Y}$ is $L$-stable. 
Then there is a map $P\dgrm{X}\to R\dgrm{Y}$, that is a projective weak equivalence by our assumption on $\dgrm{Y}$ and that makes the following diagram commutative:
\diagr{ & \dgrm{X} \ar[dl]\ar[d]\ar[r] & \dgrm{Y} \\
        P'\dgrm{X} \ar[r]^\simeq & P\dgrm{X} \ar[ur]_{\simeq} & } 
Since $P'\dgrm{X}$ is projectively cofibrant, the adjoint map $LP'\dgrm{X}\to\dgrm{Y}$ is a projective weak equivalence and hence $LP'\dgrm{X}$ is linear. The map $L\dgrm{X}\to LP'\dgrm{X}$ is an $L$-stable acyclic cofibration, since $\dgrm{X}\to P'\dgrm{X}$ is one and $L$ preserves them. Therefore the induced map $P L\dgrm{X}\to LP'\dgrm{X}$ is a projective weak equivalence. We finally find, that $L\dgrm{X}\to\dgrm{Y}$ is an $L$-stable equivalence. The converse direction is similar.  
\end{proof}

\begin{remark}\label{assembly is L-we}
If \dgrm{X} is an arbitrary functor in $\mc{M}^{\mc{U}}$, then the assembly map
    $$ \dgrm{X}(\frei)\otimes L\to\dgrm{X}(\frei\otimes L) $$
is an $L$-stable equivalence.: its adjoint is an $L$-stable equivalence, since the maps in the homotopy colimit for $P_L$ factor over each other, now the assertion follows from the previous theorem.
\end{remark}

\section{Functoriality}
In this section we study functoriality of our $L$-stabilization or $L$-stabilization. We first consider the construction with \mc{M} as variable. Then we prove independence from the weak homotopy type of $L$. Finally we observe, how the functor categories behave with respect to the source \mc{U}.

The tensor is a functor $\mc{M}\times\mc{V}\to\mc{M}$, that restricts to a functor $\mc{M}\times\mc{U}\to\mc{M}$. This induces a functor $\lambda\co\mc{M}\to\mc{M}^\mc{U}$ given by
\begin{equation}\label{lambda}
      \lambda(M)=M\otimes\id_{\mc{U}}=M\otimes R^S .  
\end{equation}
Here we view the identity functor of \mc{U} as the representable functor associated to the unit $S$ restricted to \mc{U}. There is a right adjoint $\rho\co\mc{M}^{\mc{U}}\to\mc{M}$, which is evaluating at $S$:
\begin{equation}\label{rho}
       \rho(\dgrm{X})=\dgrm{X}(S)\cong\Vnat(R^S,\dgrm{X})
\end{equation}
Remember, that this is well defined, since by axiom \ref{U1}(4)\, \mc{U} contains the unit $S$.
Recall also, that the action of $L$ on $\mc{M}^{\mc{U}}$ is given by the tensor in \ref{V + small} and that we have written:
    $$ L\dgrm{X}=L\circ\dgrm{X}=\dgrm{X}\otimes L$$
Observe, that for an object $K$ in \mc{V} we have
    $$ \lambda(LM)(K)\cong M\otimes L\otimes K,$$
while
    $$ (L\circ\lambda(M))(K)\cong M\otimes K\otimes L\cong(\lambda(M)\circ L)(K).$$
So these functors are naturally isomorphic, but this isomorphism involves the twist map of our symmetric monoidal structure. We have also isomorphisms
    $$ R\rho(\dgrm{X})\cong\dgrm{X}(S)^L\cong\rho(R\dgrm{X}). $$
We summarize:
\begin{lemma}\label{restricted lambda-rho}
For any \mc{U} satisfying \emph{\ref{U1}} the pair $(\lambda,\rho)$ constitutes a morphism 
\begin{equation*}
     \lambda\co\mc{M}\rightleftarrows\mc{M}^{\mc{U}}:\!\rho 
\end{equation*}
in the category $\ulvmod$ of \mc{V}-model categories \emph{\ref{ulvmod}}. 
\end{lemma}

To study functoriality of our $L$-stabilization 
let $\Phi\co\mc{M}_1\to\mc{M}_2$ be a morphism in the category $\ulvmod$. So $\Phi$ is a left \mc{V}-Quillen functor with right adjoint $\Gamma$. There are induced 2-functors
    $$ \Phi_*\co\mc{M}_1^{\mc{U}}\to\mc{M}_2^{\mc{U}} \text{ and }\Gamma_*\co\mc{M}_2^{\mc{U}}\to\mc{M}_1^{\mc{U}}, $$
defined by postcomposing with $\Phi$ and $\Gamma$.  
Note, that $\Phi_*\dgrm{X}$ is small as long as \dgrm{X} is so, because $\Phi$ as a left \mc{V}-Quillen functor commutes with \mc{V}-left Kan extensions. Postcomposing with $\Gamma$ does not necessarily land in small functors again, therefore we have to assume, that \mc{U} is small to get a satisfactory theory. Postcomposing with a \mc{V}-functor preserves preserves all kinds of composition laws of 1- and 2-morphisms in the category $\ulvmod$, which means, that the associations
    $$ \stab_L^{\mc{U}}\co\mc{M}\mapsto\mc{M}^{\mc{U}} \text{ and } (\Phi,\Gamma)\mapsto(\Phi_*,\Gamma_*)$$
constitute a 2-functor. 



\begin{theorem}\label{M-functoriality}
Let us fix \mc{V} and a small \mc{U} satisfying \emph{\ref{U1}}. Then for each small and cofibrant $L$ we can associate to \mc{M} in \mc{V}{\rm -mod} the category $\mc{M}^{\mc{U}}$ equipped with the $L$-stable model structure. This constitutes a $2$-functor from the category $\ulvmod$ \emph{\ref{ulvmod}} to itself. This $2$-functor preserves Quillen equivalences. 
\end{theorem}

\begin{proof}
$2$-functoriality is just clear. 

Obviously, $\Gamma_*$ preserves objectwise weak equivalences and fibrations, so it is a right Quillen functor for the projective model structures on both sides. Now in all the localizations acyclic fibrations do not change, and weak equivalences between fibrant objects are objectwise weak equivalences. So $\Gamma_*$ preserves them, and by lemma \ref{Dugger} it is a right Quillen functor for the homotopy and the $L$-stable model structure. $\Phi_*$ and $\Gamma_*$ are obviously \mc{V}-functors. 

Finally, if $(\Phi,\Gamma)$ form a Quillen equivalence, then $(\Phi_*,\Gamma_*)$ is a Quillen equivalence for the projective model structures. This descends obviously to the homotopy model structures, since $\Phi_*(\dgrm{X}\fibfun)\cong\Phi_*(\dgrm{X})\fibfun$. For the $L$-stable model structures we have to check by \ref{Dugger}, that $\Phi_*$ preserves $L$-stable fibrations between $L$-stable functors. Those are just objectwise fibrations, which $\Phi_*$ indeed preserves. 
\end{proof}

We have used the following lemma due to Dugger \cite{Dugger:replace}. A proof is also found in \cite[Prop. 8.6.4.]{Hir:loc}
\begin{lemma}\label{Dugger}
For a pair $(L,R)$ of adjoint functors the following are equivalent:
\begin{punkt}
    \item $(L,R)$ is a Quillen pair.
    \item The left adjoint $L$ preserves cofibrations between cofibrant objects and all acyclic cofibrations.
    \item The right adjoint $R$ preserves fibrations between all fibrant objects and all acyclic fibrations.
\end{punkt} 
\end{lemma}

\begin{remark}
For a morphism $\Phi\co\mc{M}\to\mc{N}$ in \mc{V}-mod and an object $\dgrm{X}$ in $\mc{M}^{\mc{U}}$ there are natural transformations
    $$ \Phi_*(\dgrm{X}(\frei))\otimes L\to\Phi_*(\dgrm{X}(\frei)\otimes L)\to\Phi_*(\dgrm{X}(\frei\otimes L)), $$
that are $L$-stable equivalences in $\mc{N}^{\mc{U}}$. The proof is straightforward, if one uses the twist map $L_1\otimes L_2\to L_2\otimes L_1$ correctly. 
\end{remark}

Next we look at $L$ as a variable.
The following theorem is the only reason we introduced the $L$-$L'$-stable model structure for different $L$ and $L'$.
\begin{theorem}\label{L-functoriality2}
If the map $L'\to L$ is a weak equivalence, the induced maps of coaugmented linearizations given by the identity functor on $\mc{M}^{\mc{U}}$ corresponding to the square \emph{\ref{L-functoriality1}} are Quillen equivalences.
In particular we obtain, that the Quillen equivalence class of $\mc{M}^{\mc{U}}$ with the $L$-stable model structure only depends on the weak homotopy type of $L$.
\end{theorem}

\begin{proof}
This is obvious.
\end{proof}

Now let $i\co\mc{U}_1\inj\mc{U}_2$ be a full inclusion of full subcategories of \mc{V}.
We observed in \ref{LKanii}, that we have an adjoint pair
\begin{equation}\label{LKan adjunction} 
    \LKan_i\co\mc{M}^{\mc{U}_1}\rightleftarrows\mc{M}^{\mc{U}_2}:\!i^*,
\end{equation}
where $i^*$ is the restriction functor and $\LKan_i$ is the \mc{V}-left Kan extension along $i$.

\begin{lemma}\label{LKan left Quillen}
The pair from \emph{(\ref{LKan adjunction})} is a \mc{V}-Quillen pair for the projective, the homotopy and the $L$-stable model structures.
\end{lemma}

\begin{proof}
We just have to observe, that the the fibrancy conditions in all the model structure are given objectwise, so $i^*$ preserves them, and hence it is a right Quillen functor.
\end{proof}

Observe, that under certain circumstances $\LKan_i$ preserves weak equivalences. Compare \ref{LKan preserves we} and \ref{LKan preserves L-we}. The question, when this adjoint functor pair is a Quillen equivalence, will be briefly addressed at the end of section \ref{section:L-stabilization}.

\section{$L$-stabilization}
\label{section:L-stabilization}

We will now define our $L$-stabilization or $L$-stabilization of a \mc{V}-model category \mc{M} from $\ulvmod$ with respect to $L$, prove an idempotency result \ref{S-eval}. This is connected with $L$-spectra and we will examine this more closely in section \ref{section:L-spectra}.

Before we prove theorem \ref{S-eval}, let us observe a few things.
Assume that $(L,R)$ is already a Quillen equivalence.
Then for all cofibrant $M$ and fibrant $N$ in \mc{V} $LM=M\otimes L\to N$ is a weak equivalence if and only if $M\to RN=N^L$ is a weak equivalence. The maps $LK\to L(K^\fib)$ and $LK\to(LK)^\fib$ are also acyclic cofibrations. It follows, that the functorial maps
    $$ K^\fib\to RL(K^\fib)\to R(L(K^\fib)^\fib) $$
are weak equivalences. We can rephrase this by saying that the projectively fibrant replacement $\fibfun$ of the identity functor of \mc{V} is $L$-stable. In fact, the identity itself is $L$-stable.

\begin{definition}\label{L}
Consider the set of objects $\{L^nS\,|\,n\in\mathbbm{N}\}$ and let $\vsph_L$ be the full subcategory of \mc{M} obtained from this set and its image under the functor $\fibfun$. Here we obviously set $L^0S=S$. For a \mc{V}-model category \mc{M} its {\it $L$-stabilization} is the category $\mc{M}^{\vsph_L}$ equipped with the $L$-stable model structure.
\end{definition}

\begin{remark}
\begin{punkt}
   \item
The category $\vsph_L$ is the smallest allowable choice for \mc{U} according to the axiom \ref{U1}.  At the end of this section we will address briefly the question, whether one can use larger categories.
   \item
If \dgrm{X} is an $L$-stable functor, then the behavior of \dgrm{X} on $\vsph_L$ is determined up to weak equivalence by its value on $S$, because the maps
\begin{equation}\label{determined by S}
  L\dgrm{X}(L^nS)\to(\dgrm{X}(L^{n+1}S)^\fib)^\fib\leftarrow \dgrm{X}(L^{n+1}S)
\end{equation}
are weak equivalences. 
\end{punkt}
\end{remark}

\begin{theorem}\label{S-eval}
Let \mc{M} be $L$-stable, see \emph{\ref{L-stable}}, i.e. $(L,R)$ is a Quillen equivalence on \mc{M}. Then we have:
\begin{punkt}
   \item
Let $\dgrm{X}\co\mc{U}\to\mc{M}$ be a linearly fibrant functor. If $\lambda(M)\to\dgrm{X}$ is an $L$-equivalence, then $M\to\dgrm{X}(S)$ is a weak equivalence.
   \item
The pair $(\lambda,\rho)$ from \emph{\ref{restricted lambda-rho}}, where $\mc{U}=\vsph_L$, is a \mc{V}-Quillen equivalence.
\end{punkt}
\end{theorem}

\begin{proof}
Let $(L,R)$ be a Quillen equivalence.
The functor $\lambda(M)$ is a projectively cofibrant $L$-stable homotopy functor. Indeed, the map
    $$ K\otimes M\otimes L\to K\otimes L\otimes M\to (K\otimes L)^\fib\otimes M\to((K\otimes L)^\fib\otimes M)^\fib $$
is a weak equivalence for every $K$ in \mc{V}. Remember that all objects in \mc{U} are cofibrant. So the adjoint map $K^\fib\to R(LK)^\fib$ is a weak equivalence. So for every cofibrant $M$, every $L$-stably fibrant $\dgrm{X}$ and every map $\lambda(M)\to\dgrm{X}$ there is a commuting diagram
\diagr{\lambda(M) \ar[r]\ar[d] & \dgrm{X} \ar[d] \\ P\lambda(M) \ar[r] & P\dgrm{X} }
consisting entirely of objectwise weak equivalences if and only if the map $\lambda(M)\to\dgrm{X}$ is an $L$-stable equivalence. So this is equivalent to $M\to\dgrm{X}(S)$ being a weak equivalence. This proves (i) and the first part of (ii). The second part of (ii) is obvious by (\ref{determined by S}).
\end{proof}

We saw, that $\vsph_L$ is the smallest choice for \mc{U}. The question, whether we can choose larger source categories, for which the previous statements are still true, is an interesting one. An indication, what we can hope for, is given by the assumptions \ref{U3}, we have to make for the monoid axiom.
The largest choice for \mc{U} in this setting is the full subcategory of \mc{V} given by the cofibrant \mc{V}-finitely presentable objects. Let us call this choice $\mc{U}_{{\rm max}}$. In the case of $\mc{V}=\mc{S}_*$ and $L=S^1$ the maximal choice is the category of finite pointed simplicial sets $\mc{S}_*^{{\rm fin}}$ and it follows from \cite{Lydakis}, that restriction and left Kan extension give a Quillen equivalence
    $$ \mc{S}_*^{\mc{S}_*^{{\rm fin}}}\leftrightarrows\mc{S}_*^{\vsph_{S^1}} .$$
This is just another formulation of the well-known fact, that a generalized homology theory is completely determined by its values on the spheres.
It is not always true, that $\mc{U}_{{\rm max}}$ and $\vsph_L$ yield Quillen equivalent functor categories in general. For a counter example see \cite[p. 465]{DRO:enriched} with $\mc{V}=\mc{S}_*$ and $L=S^0\sqcup S^0$. So we would like to pose the following question.

\begin{question}
Given \mc{V} and $L$.
What is the largest full subcategory \mc{U} of \mc{V}, such that the pair
    $$ \LKan_i\co\mc{M}^{\vsph_L}\rightleftarrows\mc{M}^{\mc{U}}:\! i^*$$
is a \mc{V}-Quillen equivalence, where $i\co\vsph_L\to\mc{U}$ is a full inclusion and where both sides carry the $L$-stable model structure.
\end{question}

The answer to this question for the motivic case with $L=\mathbbm{P}^1$ is open and is likely to involve Voevodsky's slice filtration.

\section{Symmetric monoidal structures}
\label{section:sym. mon. str.}

We are now going to describe, how to enrich the category $\mc{M}^{\mc{U}}$ over $\mc{V}^{\mc{U}}$. In particular, for the case $\mc{M}=\mc{V}$ the category $\mc{V}^{\mc{U}}$ will become again a closed symmetric monoidal category. This enrichment is a straightforward generalization of the smash product constructed by Lydakis in \cite{Lydakis} for the case $\mc{V}=\mc{S}_*$ of pointed simplicial sets and $\mc{U}=\mc{S}_*^{\text{fin}}$ of finite pointed simplicial sets. 

\begin{axiom}\label{U2}
Additionally to \ref{U1} we now have to make the following assumption.
\begin{enumerate}\setrefstep{5}
    \item
\mc{U} is closed under the symmetric monoidal product $\otimes$.
\end{enumerate}
\end{axiom}

First we define a new category $\mc{U}\otimes\mc{U}$:
\begin{align*}
    {\rm Ob}(\mc{U}\otimes\mc{U})&={\rm Ob}(\mc{U})\times{\rm Ob}(\mc{U}) \\
    {\rm Mor}(\mc{U}\otimes\mc{U})((A,B),(C,D))&=\mc{V}(A,C)\otimes\mc{V}(B,D) 
\end{align*}
Given two functors $\dgrm{U}\co\mc{U}\to\mc{V}$ and $\dgrm{X}\co\mc{U}\to\mc{M}$ there is a functor
   $$ \dgrm{X}\varodot\dgrm{U}\co\mc{U}\otimes\mc{U}\to\mc{M} $$
defined by
\begin{equation}\label{varodot} 
   (\dgrm{X}\varodot\dgrm{U})(A,B)=\dgrm{X}(A)\otimes\dgrm{U}(B) .
\end{equation}
Further the monoidal product $\otimes\co\mc{U}\times\mc{U}\to\mc{U}$ factors over $\mc{U}\otimes\mc{U}$ to give a functor
   $$ \otimes\co\mc{U}\otimes\mc{U}\to\mc{U}, (A,B)\mapsto A\otimes B. $$
The {\it tensor product} of \dgrm{U} and \dgrm{X} is then defined as the left Kan extension of $\dgrm{X}\varodot\dgrm{U}$ along $\otimes$:
\diagr{ \mc{U}\otimes\mc{U} \ar[r]^-{\dgrm{X}\varodot\dgrm{U}}\ar[d]_{\otimes} & \mc{M}  \\
        \mc{U} \ar@{.>}[ur]_-{\dgrm{X}\otimes\dgrm{U}} }
Observe first, that this Kan extension exists for representable functors $\dgrm{U}=R^A\otimes M$ and $\dgrm{X}=R^B$ for objects $A,B$ in \mc{U} and $M$ in \mc{M}. There it is given simply by
   $$ (R^A\otimes M)\otimes R^B=R^{A\otimes B}\otimes M .$$
But now we can extend it to all small functors by using the isomorphism from remark \ref{properties of smallness}. The resulting functors are clearly small again. 

\begin{definition}\label{M^V over itself}
For two small functors $F,G\co\mc{U}\to\mc{M}$ we define a new functor $\mc{M}^{\mc{U}}(F,G)\co\mc{U}\to\mc{V}$ in the following way:
    $$ \mc{M}^{\mc{U}}(F,G)(K):=\Vnat(F,G(K\otimes\frei)) $$
There is also a {\it cotensor} $(\mc{V}^{\mc{U}})^{\rm op}\times\mc{M}^{\mc{U}}\to\mc{M}^{\mc{U}}$ given by 
    $$ H^G(K):=\int_{K\in\mc{K}}H(K\otimes\frei)^{G(K)},$$
where on the right hand side the cotensor is the underlying \mc{V}-cotensor from \ref{V + small} and \mc{K} is a defining subcategory for $G$. 
\end{definition}

These functors endow $\mc{M}^{\mc{U}}$ with a tensor, cotensor and an enrichment over $\mc{V}^{\mc{U}}$.
\begin{lemma}\label{adjunctions}
For small functors $F,H\co\mc{U}\to\mc{M}$ and $G\co\mc{U}\to\mc{V}$ we have:
    $$ \mc{M}^{\mc{U}}(F,H^G)\cong\mc{M}^{\mc{U}}(F\otimes G,H)\cong\mc{M}^{\mc{U}}(F,H)^G $$
\end{lemma}

This proof is considerably easier than \cite[Prop. 5.2]{Lydakis}. 
\begin{proof}
Let $F$ and $G$ be given by the following colimits:
   $$ F\cong\int^AR^A\otimes FA \text{ and } F\cong\int^BR^B\otimes GB $$
We have the following isomorphisms:
\begin{align*}
    \mc{M}^{\mc{U}}(F,H^G) &\cong \int^A\mc{M}^{\mc{U}}(R^A\otimes FA, K\mapsto H^G(K\otimes\frei)) \\
    & \cong \int^A\int^BH(A\otimes B)^{FA\otimes GB} \\
    & \cong \int^A\int^B\mc{M}^{\mc{U}}(R^{A\otimes B}\otimes FA\otimes GB,H) \\         
    & \cong \mc{M}^{\mc{U}}(F\otimes G,H)    
\end{align*}
The second isomorphism is similar.
\end{proof}

\begin{theorem}
Let \mc{V} is a closed symmetric monoidal model category. Then the category $\mc{V}^{\mc{U}}$ can be given the structure of a closed symmetric monoidal category. If \mc{M} is \mc{V}-category, then $\mc{M}^{\mc{U}}$ can be given the structure of a $\mc{V}^{\mc{U}}$-category. The functor 
    $$ \lambda\co\mc{M}\to\mc{M}^{\mc{U}} $$
from \emph{\ref{restricted lambda-rho}} is the left adjoint of a symmetric monoidal \mc{V}-adjunction. 
\end{theorem}

\begin{proof}
It follows from lemma \ref{adjunctions}, that $\mc{V}^{\mc{U}}$ is closed symmetric monoidal and that $\mc{M}^{\mc{U}}$ is a $\mc{V}^{\mc{U}}$-category. It is clear, that $(\lambda,\rho)$ form a \mc{V}-adjunction. For objects $M$ and $N$ in \mc{M} we also have the following natural isomorphisms:
\begin{align*}
   \lambda(M\otimes N)&=(M\otimes N)\otimes R^S\cong (M\otimes N)\otimes R^{S\otimes S}\cong (M\otimes R^S)\otimes(N\otimes R^S) \\
                      &=\lambda(M)\otimes\lambda(N)
\end{align*}
These isomorphism commute with the twist map, which proves, that $\lambda$ is a symmetric monoidal functor.
\end{proof}

The unit of the monoidal structure on $\mc{V}^{\mc{U}}$ is given by the inclusion functor $R^S\co\mc{U}\inj\mc{V}$, which for $\mc{U}=\mc{V}$ is, of course, the identity. Observe, that the unit of $\mc{V}^{\mc{U}}$ is always projectively cofibrant by \ref{rep. cof}, because the unit $S$ is cofibrant.

In order to prove compatibility of the projective model structure of $\mc{M}^{\mc{U}}$ with its $\mc{V}^{\mc{U}}$-enrichment, we have to remember, that \mc{V} and \mc{M} were assumed to be cofibrantly generated with generating sets $I_{\mc{V}}, J_{\mc{V}}$ and $I_{\mc{M}}, J_{\mc{M}}$ respectively. 
Here, as everywhere, the $I$'s are used for the cofibrations and the $J$'s for the acyclic cofibrations.
It is shown in \cite{Chorny-Dwyer:small}, that $\mc{V}^{\mc{U}}$ and $\mc{M}^{\mc{U}}$ are class-cofibrantly generated with generating sets 
    $$I_{\mc{M}^{\mc{U}}}=\{R^A\otimes i\,|\,i\in I_{\mc{M}}, A\in{\rm Ob}(\mc{U})\} $$
and 
    $$J_{\mc{M}^{\mc{U}}}=\{R^A\otimes j\,|\,j\in J_{\mc{M}}, A\in{\rm Ob}(\mc{U})\},$$
where $\mc{M}=\mc{V}$ is a special case.
Remember again, that here we are using the generalized small object argument \cite{Chorny:general} to construct (non-functorial) factorizations, which accepts locally small classes rather than sets as input.

\begin{theorem}\label{tensor + projective}
The category $\mc{M}^{\mc{U}}$ with its projective model structure is a $\mc{V}^{\mc{U}}$-model category, where $\mc{V}^{\mc{U}}$ has the projective model structure.
\end{theorem}

\begin{proof}
Using lemma \ref{smash compatible and generators} we have to check that for arbitrary objects $K_1$ and $K_2$ in \mc{U} and arbitrary maps $i'_1$ in $I_{\mc{V}}$ and $i'_2$ in $\mc{I}_{\mc{M}}$ the pushout product of $i_1=R^{K_1}\otimes i'_1\co R^{K_1}\otimes A\to R^{K_1}\otimes B$ and $i_2=R^{K_2}\otimes i'_2\co R^{K_2}\otimes C\to R^{K_2}\otimes D$ given by 
    $$ i_1\,\square\, i_2\co (R^{{K_1}\otimes {K_2}}\otimes A\otimes D)\sqcup_{(R^{{K_1}\otimes {K_2}}\otimes A\otimes C)}(R^{{K_1}\otimes {K_2}}\otimes B\otimes C)\to R^{{K_1}\otimes {K_2}}\otimes B\otimes D $$
is a projective cofibration. So, let us map it into an arbitrary acyclic projective fibration $\dgrm{X}\to\dgrm{Y}$. But in the diagram
\diagr{ (A\otimes D)\sqcup_{(A\otimes C)}(B\otimes C)\ar[r]\ar[d] & \dgrm{X}({K_1}\otimes {K_2}) \ar[d] \\
        B\otimes D \ar[r] & \dgrm{Y}({K_1}\otimes {K_2})}
there exists a lifting, because the left hand side is a cofibration and the right hand side is an acyclic fibration. This diagram is adjoint to our original lifting problem and proves the first part. 
The remaining properties to be shown are similar.
\end{proof}

\begin{theorem}
Let \mc{V} be right proper. The category $\mc{M}^{\mc{U}}$ with its homotopy model structure is a $\mc{V}^{\mc{U}}$-model category, where $\mc{V}^{\mc{U}}$ has either the projective or the homotopy model structure.
\end{theorem} 

\begin{proof}
Cofibrations do not change. It remains to show, that ${\rm fib}\boxtimes I_{\mc{V}^{\mc{U}}}\subset{\rm fib}$ and $I_{\mc{M}^{\mc{U}}}\,\square\, {\rm acof}\subset{\rm acof}$, where ${\rm fib}$ and ${\rm acof}$ are the classes of fibrations and acyclic co\-fibra\-tions in the homotopy model structure on $\mc{M}^{\mc{U}}$ and ${\rm acof}$ is the class of acyclic co\-fi\-bra\-tions in the homotopy model structure on $\mc{V}^{\mc{U}}$.

Let $\dgrm{X}\to\dgrm{Y}$ be a fibration in the homotopy model structure and let $i\co E\to F$ be in $\mc{I}_{\mc{V}}$. We have to show, that for every object $A$ in \mc{U} the map
\begin{equation}\label{map1}
  \dgrm{X}^{R^A\otimes F}\to\dgrm{Y}^{R^A\otimes F}\times_{\dgrm{Y}^{R^A\otimes E}}\dgrm{X}^{R^A\otimes E} 
\end{equation}
is a fibration in the homotopy model structure. We know already by \ref{tensor + projective}, that it is an objectwise fibration. Precomposing with the functor $\fibfun$ commutes with limits of the functor category. So it remains to show, that the diagram
\diagr{ \dgrm{X}^{R^A\otimes F} \ar[r]\ar[d] & \dgrm{X}^{R^A\otimes F}\fibfun \ar[d] \\
        \dgrm{Y}^{R^A\otimes F}\times_{\dgrm{Y}^{R^A\otimes E}}\dgrm{X}^{R^A\otimes E} \ar[r]     & (\dgrm{Y}^{R^A\otimes F}\fibfun)\times_{(\dgrm{Y}^{R^A\otimes E}\fibfun)}(\dgrm{X}^{R^A\otimes E}\fibfun)}
is an objectwise homotopy pullback diagram. 
The functor $\dgrm{X}^{R^A\otimes F}\co\mc{U}\to\mc{M}$ is given by
     $$ K\mapsto \dgrm{X}(K\otimes A)^F.$$ 
So in the following diagram for every $K$ and $A$ the first square has to be a homotopy pullback square.
\diag{ \dgrm{X}^F(K\otimes A) \ar[r]\ar[d]\ar@{}[dr]|{1} &  \dgrm{Y}^F(K\otimes A)\times_{\dgrm{Y}^{E}(K\otimes A)}\dgrm{X}^{E}(K\otimes A) \ar[d]   \\
        \dgrm{X}^{F}(K^\fib\otimes A) \ar[r]\ar[d] \ar@{}[dr]|{2} & \dgrm{Y}^{F}(K^\fib\otimes A)\times_{\dgrm{Y}^{E}(K^\fib\otimes A)}\dgrm{X}^{E}(K^\fib\otimes A) \ar[d]  \\ 
        \dgrm{X}^{F}(K^\fib\otimes A)^\fib \ar[r] & \dgrm{Y}^{F}(K^\fib\otimes A)^\fib\times_{\dgrm{Y}^{E}(K^\fib\otimes A)^\fib}\dgrm{X}^{E}(K^\fib\otimes A)^\fib }{alotofpullbacks1}
Since the map $\dgrm{X}^F\to\dgrm{Y}^F\times_{\dgrm{Y}^E}\dgrm{X}^E$ is a fibration in the homotopy model structure by \ref{V + homotopy}, the second square is a homotopy pullback. 
Using the fact, that $\fibfun$ assigns a weak equivalence $Z\to Z^\fib$ to every $Z$ and that every object in \mc{U} is cofibrant, we can construct a weak equivalence $(K\otimes A)^\fib\simeq(K^\fib\otimes A)^\fib$ under $K\otimes A$. All our functors preserve this weak equivalence, because both objects are fibrant. Hence the outer square in diagram (\ref{alotofpullbacks1}) is a homotopy pullback square, since it is weakly equivalent to the following homotopy pullback square:
\diag{ \dgrm{X}^F(K\otimes A) \ar[r]\ar[d] &  \dgrm{Y}^F(K\otimes A)\times_{\dgrm{Y}^{E}(K\otimes A)}\dgrm{X}^{E}(K\otimes A) \ar[d]   \\
        \dgrm{X}^{F}(K\otimes A)^\fib \ar[r] & \dgrm{Y}^{F}(K\otimes A)^\fib\times_{\dgrm{Y}^{E}(K\otimes A)^\fib}\dgrm{X}^{E}(K\otimes A)^\fib }{alotofpullbacks2}
Then square 1 of (\ref{alotofpullbacks1}) is a homotopy pullback by a standard argument on homotopy pullbacks.
This proves ${\rm fib}\boxtimes I_{\mc{V}^{\mc{U}}}\subset{\rm fib}$.

Now for $I_{\mc{M}^{\mc{U}}}\,\square\, {\rm acof}\subset{\rm acof}$. 
Let $j\co\dgrm{E}\to\dgrm{F}$ be an acyclic cofibration in the homotopy model structure on $\mc{V}^\mc{U}$. By \ref{fib + we} this implies, that the restriction of $j$ to the category of fibrant objects is an acyclic projective cofibration. It follows from \ref{tensor + projective}, that the restriction of the map $i\,\square\,j$ to the category of fibrant objects is an acyclic projective cofibration. This proves, that $i\,\square\,j$ is acyclic cofibration in the homotopy model structure.
\end{proof}

\begin{definition}
Let \mc{M} be a model category and a \mc{V}-category for a closed symmetric monoidal model category \mc{V} with cofibrant unit. Then \mc{M} is said to be a {\it semi-\mc{V}-model category}, if the following properties are satisfied:
\begin{punkt}
    \item ${\rm cof}_{\mc{M}}\,\square\,{\rm cof}_{\mc{V}}\subset{\rm cof}_{\mc{M}}$
    \item ${\rm acof}_{\mc{M}}\,\square\,{\rm cof}_{\mc{V}}\subset{\rm acof}_{\mc{M}}$
    \item The map $i\,\square\, j$ is an acyclic cofibration for every cofibration $i$ in \mc{M} and every acyclic cofibration $j$ between fibrant objects in \mc{V}. 
\end{punkt}
Here cof and acof denote the classes of cofibrations and acyclic cofibrations with the respective category as subscript.
\end{definition} 

\begin{theorem}\label{tensor + L-stable}
Let \mc{V} be right proper.
The category $\mc{M}^{\mc{U}}$ with its $L$-stable model structure is a $\mc{V}^{\mc{U}}$-model category, where $\mc{V}^{\mc{U}}$ has either the projective or the homotopy model structure. It is a semi-$\mc{V}^{\mc{U}}$-model category, if we equip $\mc{V}^{\mc{U}}$ with the $L$-stable model structure.
If \mc{M} is left proper and all maps in $I_{\mc{M}}$ have cofibrant source, we have a full $\mc{V}^{\mc{U}}$-model category. 
\end{theorem}

\begin{proof}
We will show first ${\rm fib}\boxtimes I_{\mc{V}^{\mc{U}}}\subset{\rm fib}$, where $\fib$ is the class of $L$-stable fibrations. A map $p\co\dgrm{X}\to\dgrm{Y}$ in $\mc{M}^{\mc{U}}$ is an $L$-stable fibration if and only if $p$ is a homotopy fibration, such that the square
\diag{ \dgrm{X} \ar[r]\ar[d] & R\dgrm{X}^hL=:T\dgrm{X} \ar[d] \\
       \dgrm{Y} \ar[r] & R\dgrm{Y}^hL=:T\dgrm{Y} }{T-pullback}
is a homotopy pullback square in the homotopy structure. Let $i\co R^A\otimes K_1\to R^A\otimes K_2$ be in $I_{\mc{V}^{\mc{U}}}$. 
Then the map
    $$ \dgrm{X}^{K_2}\to\dgrm{Y}^{K_2}\times_{\dgrm{Y}^{K_1}}\dgrm{X}^{K_1}$$
is an $L$-stable fibration by \ref{V + homotopy}. Now cotensoring with $R^A$ commutes with limits, as well as with $T$, since the right adjoint $R$ commutes with cotensors. We deduce, that for the map
    $$ p\boxtimes i\co\dgrm{X}^{R^A\otimes K_2}\to\dgrm{Y}^{R^A\otimes K_2}\times_{\dgrm{Y}^{R^A\otimes K_1}}\dgrm{X}^{R^A\otimes K_1} $$
the diagram corresponding to (\ref{T-pullback}) is a homotopy pullback. Therefore $p\boxtimes i$ is a homotopy fibration.

Now we have to show, that $i\,\square\,j$ for $i$ a projective cofibration in $\mc{M}^{\mc{U}}$ and $j$ an acyclic cofibration in the $L$-stable model structure on $\mc{V}^{\mc{U}}$ between $L$-stably fibrant objects is an $L$-stable equivalence. But we know this already, since $j$ is simply an acyclic projective cofibration.

Finally assume, that \mc{M} is left proper, and recall, that by \ref{V + L-stable} in this case the $L$-stable model structure on $\mc{M}^{\mc{U}}$ is left proper. By assumption we also have, that the source and target of the generating cofibrations $I_{\mc{M}}$ of \mc{M} can be chosen to be cofibrant. Let $j\co\dgrm{C}\to\dgrm{D}$ be an acyclic $L$-stable cofibration in $\mc{V}^{\mc{U}}$. Since \mc{M} is right proper, the functor $R^A$ is projectively cofibrant for all objects $A$ in \mc{U}. By what we have shown already, the map
    $$ R^A\otimes\dgrm{C}\to R^A\otimes\dgrm{D} $$
then is an $L$-stably acyclic cofibration in $\mc{V}^{\mc{U}}$. Tensoring with a cofibrant object $M$ or $N$ produces again an $L$-stably acyclic cofibration by \ref{V + L-stable}. The fact, that the pushout product now is an $L$-stably acyclic cofibration, follows easily from the left properness of the $L$-stable model structure.
\end{proof}

For a comment about the cofibrancy assumption in \ref{tensor + L-stable} see question \ref{cofibrancy}.
For the special case $\mc{M}=\mc{V}$ we do not need left properness.
\begin{corollary}\label{V^V symmetric monoidal}
If \mc{V} is right proper, the category $\mc{V}^{\mc{U}}$ with its $L$-stable model structure is a closed symmetric monoidal model category.
\end{corollary}

\begin{proof}
This follows from \ref{tensor + L-stable} and the fact, that the monoidal structure on $\mc{V}^{\mc{U}}$ is symmetric.
\end{proof}

\section{The monoid axiom}

The material in this section is adapted from \cite{DRO:enriched} with only minor changes.
\begin{definition}\label{R-cell}
Let \mc{C} be a model category.
For a class \mc{R} of maps in \mc{C} let \mc{R}-cell be the class of maps obtained by all transfinite compositions of cobase changes of maps in \mc{R}. The precise definition can be found in \cite[10.5.8]{Hir:loc}.

Now let \mc{C} be a symmetric monoidal model category with monoidal product $\otimes$.
If \mc{O} is a class of objects in \mc{C} we denote by $\mc{R}\otimes\mc{O}$ all maps of the form 
    $$ r\otimes\id_X\co A\otimes X\to B\otimes X, $$
where $r$ runs through \mc{R} and $X$ trough \mc{O}. 
\end{definition}

We remind the reader that ${\rm acof}_{\mc{C}}$ stands for the class of acyclic cofibrations in \mc{C}. The following axiom was introduced by Schwede and Shipley in \cite{SS:monoid} in order to study the homotopy theory of monoids and algebras and modules over them in a monoidal model category.
\begin{axiom}
The model category \mc{C} satisfies the monoid axiom, if all maps in $\{{\rm acof}_{\mc{C}}\otimes{\rm Ob}(\mc{C})\}$-cell are weak equivalences.
\end{axiom}

\begin{remark}
\begin{punkt}
   \item
If the model category \mc{C} is cofibrantly generated with generating set $J_{\mc{C}}$ of acyclic cofibrations, then the monoid axiom holds in \mc{C}, if all maps in $J_{\mc{C}}\otimes{\rm Ob}(\mc{C})$ are weak equivalences. This follows, since in this case all acyclic cofibrations are retracts of maps in $J_{\mc{C}}$-cell.
   \item
The monoid axiom is trivially satisfied, if all objects in \mc{C} are cofibrant. 
\end{punkt}
\end{remark}

The monoid axiom is easily proved for the projective and the homotopy model structure.
We get into trouble to prove it for the $L$-stable model structure, where we have to introduce a lot of very technical conditions.
\begin{theorem}
If \mc{V} satisfies the monoid axiom, the projective model structure on $\mc{V}^{\mc{U}}$ satisfies the monoid axiom.
\end{theorem}

\begin{proof}
We have to show, that the maps in the class 
   $$\{R^U\otimes J_{\mc{V}}|U\in{\rm Ob}(\mc{U}\}\otimes{\rm Ob}(\mc{V}^{\mc{U}}){\rm -cell}$$ 
are weak equivalences. We can check the monoid axiom for every \dgrm{X} in $\mc{V}^{\mc{U}}$ at a time. 
The functors $R^U$ are small in the sense, that mapping out of them commutes with filtered colimits, so compositions of objectwise weak equivalences always stay weak equi\-va\-lences. 
We can then check the axiom objectwise, that means, that the class above is equal to the following class: 
   $$\bigcup_{\dgrm{X}\in{\rm Ob}(\mc{V}^{\mc{U}})}\bigcup_{K\in{\rm Ob}(\mc{U})} \{R^U\otimes J_{\mc{V}}|U\in{\rm Ob}(\mc{U}\}\otimes(\dgrm{X}(K)){\rm -cell} $$ 
Now the monoid axiom in $\mc{V}^{\mc{U}}$ follows from the monoid axiom in \mc{V}.
\end{proof}

For the homotopy model structure we have to assume a small \mc{U}, because we will check the monoid axiom on the generating acyclic cofibrations and we do not know, whether the homotopy model structure for non-small \mc{U} is class-cofibrantly generated. 
\begin{theorem}
If \mc{V} is satisfies the monoid axiom and \mc{U} is small, the homotopy model structure on $\mc{V}^{\mc{U}}$ satisfies the monoid axiom.
\end{theorem}

\begin{proof}
If \mc{U} is small, the homotopy model structure is cofibrantly generated with the class $J_{\mc{V}^{\mc{U}}}^{\rm ho}=J''\cup J_{\mc{V}^{\mc{U}}}^{\rm proj}$ given in \ref{J^ho} as generating acyclic cofibrations. Since we already know, that the projective model structure satisfies the monoid axiom, it suffices to check, that the class 
   $$ J''\otimes{\rm Ob}(\mc{V}^{\mc{U}}){\rm -cell} $$
consists of weak equivalences in the homotopy model structure. So let $f\co A\to B$ be an acyclic cofibration in \mc{U} and factor the induced map $f^*\co R^B\to R^A$ as in \ref{J^ho}: 
\diagr{ R^B \ar[r]^-{j} & \cyl(f^*) \ar[r]^-{p} & R^A }
Let $i\co C\to D$ be a map in $I_{\mc{V}}$. 
The map
   $$ \left[(R^B\otimes D)\scup\limits_{(R^B\otimes C)}(\cyl(f^*)\otimes C)\right]\to\cyl(f^*)\otimes D  $$
is a projective cofibration and a weak equivalence, when we evaluate it on the fibrant objects of \mc{U}. So it is an acyclic projective cofibration on the full subcategory $\mc{U}^{\fib}$ of fibrant objects of \mc{U}. But then the monoid axiom for the projective structure on $\mc{V}^{\mc{U}^{\fib}}$ gives the result. 
\end{proof}

In order to prove the monoid axiom for the $L$-stable model structure we need some auxiliary facts and definitions.

\begin{definition}
Let \mc{M} be a cocomplete category and let $\Hom{\mc{M}}{\frei}{\frei}$ denote the set of morphisms in \mc{M}. An object $A$ in \mc{M} is {\it finitely presentable}, if the functor $\Hom{\mc{M}}{A}{\frei}$ commutes with all filtered colimits.

If \mc{M} is a \mc{V}-cocomplete \mc{V}-category, an object $A$ in \mc{M} is {\it \mc{V}-finitely presentable}, if the \mc{V}-Hom functor $R^A=\mc{M}(A,\frei)$ commutes with all filtered \mc{V}-colimits. 
\end{definition}

\begin{remark}
If the unit $S$ of \mc{V} is finitely presentable, then every \mc{V}-finitely presentable object is finitely presentable.
\end{remark}

\begin{lemma}[\cite{DRO:enriched} Lemma 3.5]\label{we closed under filtered colimits}
If sources and targets of the generating set $I_\mc{M}$ of cofibrations of \mc{M} are finitely presentable, the class of weak equivalences and the class of acyclic fibrations are closed under filtered colimits.
\end{lemma}

\begin{proof}
Let \mc{H} be a small index category for our colimit and equip $\mc{M}^{\mc{H}}$ with the projective model structure. Factor an objectwise weak equivalence into an acyclic projective cofibration $g$ followed by an objectwise acyclic fibration $p$. $\colim$ is a left Quillen functor, hence $\colim g$ is again an acyclic cofibration. The lemma follows, if we can prove, that $\colim p$ is an objectwise acyclic fibration for an objectwise acyclic fibration $p\co \dgrm{X}\to \dgrm{Y}$.
Consider a diagram
\diagr{ A \ar[r]\ar[d]_i & \colim\dgrm{X} \ar[d]^p \\ B \ar[r] & \colim\dgrm{Y} }
where $i\in I$. By adjointness the existence of a lift is equivalent to the surjectivity of the map
    $$ \phi\co\Hom{\mc{M}}{B}{\colim\dgrm{X}}\to\Hom{\mc{M}}{A}{\colim\dgrm{X}}\x\limits_{\Hom{\mc{M}}{A}{\colim\dgrm{Y}}}\Hom{\mc{M}}{B}{\colim\dgrm{Y}} $$
where $\Hom{\mc{M}}{\frei}{\frei}$ denotes the Hom-set of the category \mc{M}. By assumption $A$ and $B$ are finitely presentable, and filtered colimits commute with pullbacks in the category of sets. So we can squeeze the colimit out and consider the maps
    $$ \phi_h\co\Hom{\mc{M}}{B}{\dgrm{X}(h)}\to\Hom{\mc{M}}{A}{\colim\dgrm{X}(h)}\!\!\!\!\!\!\!\!\!\x\limits_{\Hom{\mc{M}}{A}{\colim\dgrm{Y}(h)}}\!\!\!\!\!\!\!\!\!\Hom{\mc{M}}{B}{\colim\dgrm{Y}(h)} $$
for each $h\in\mc{H}$. Now $\phi_h$ is surjective, since $i$ is a cofibration and $p(h)\co\dgrm{X}(h)\to\dgrm{Y}(h)$ is an acyclic fibration. But the colimit of surjective maps of sets is surjective.
\end{proof}

\begin{axiom}\label{U3}
We need two more assumptions on our categories \mc{U} and \mc{V}:
\begin{enumerate}\setrefstep{6}
    \item \label{7}
Every object in \mc{V} is a filtered colimit of objects in \mc{U}. 
    \item \label{8}
Every object of \mc{U} is \mc{V}-finitely presentable.
\end{enumerate}
\end{axiom}

Let $u\co\mc{U}\to\mc{V}$ be the inclusion functor. Then for any \mc{V}-model category \mc{M} we denote by $\LKan_u\co\mc{M}^{\mc{U}}\to\mc{M}^{\mc{V}}$ the \mc{V}-left Kan extension along $u$, compare \ref{LKan}.
\begin{lemma}[\cite{DRO:enriched} Lemma 4.9]\label{LKan preserves we}
Suppose that sources and targets in $I_\mc{M}$ are finitely presentable and that \mc{U} satisfies axioms $(7)$ and $(8)$. Then $\LKan_u$ preserves objectwise weak equivalences.
\end{lemma}

\begin{proof}
Let $f\co\dgrm{X}\to\dgrm{Y}$ be an objectwise weak equivalence in $\mc{M}^{\mc{U}}$ and let $V\in\mc{V}$. Since \dgrm{X} is small, we can write:
    $$ \LKan_uf(V)=\int^{K\in\mc{U}}R^K(V)\otimes f(K)\co\dgrm{X}(V)\to\dgrm{Y}(V)$$
By axiom (\ref{7}) it is a filtered colimit of objects in \mc{U}, so there exists a functor $C\co\mc{H}\to\mc{V}$, such that \mc{H} is filtered and $\colim_{h\in\mc{H}}C(h)\cong V$. We have with (\ref{8}):
\begin{align*}
    \LKan_uf(V)&=\int^{K\in\mc{U}}R^K(V)\otimes f(K) \\
               &\cong \int^{K\in\mc{U}}\mc{V}(K,\colim_{h\in\mc{H}}C(h))\otimes f(K) \\
               &\cong \colim_{h\in\mc{H}}\int^{K\in\mc{U}}\mc{V}(K,C(h))\otimes f(K) \\
               &\cong \colim_{h\in\mc{H}}f(C(h))
\end{align*}
So $\LKan_uf(V)$ is a colimit of weak equivalence and hence itself a weak equivalence by \ref{we closed under filtered colimits}.
\end{proof}

Curiously the following easy consequence does not appear in \cite{DRO:enriched}.
\begin{corollary}\label{LKan preserves L-we}
Suppose that sources and targets in $I_\mc{M}$ are finitely presentable and that \mc{U} satisfies axioms $(7)$ and $(8)$. Then $\LKan_u$ preserves $L$-stable equivalences.
\end{corollary}

\begin{proof}
By \ref{LKan left Quillen} the functor $\LKan_u$ is a left Quillen functor from $\mc{M}^{\mc{U}}$ to $\mc{M}^{\mc{V}}$, where both categories carry the $L$-stable model structure. If $f$ is an $L$-stable equivalence, we factor it into an $L$-stably acyclic cofibration $i$ followed by an objectwise acyclic fibration $p$. Then $\LKan_ui$ is an $L$-stably acyclic cofibration, since $\LKan_u$ is a left Quillen functor, and $\LKan_up$ is an objectwise equivalence by \ref{LKan preserves we}.
\end{proof}

\begin{definition}[\cite{DRO:enriched} Def. 4.6]
A monoidal model category \mc{V} is {\it strongly left proper} if the cobase change of a weak equivalence along any map in $\cof\otimes{\rm Ob}(\mc{V})$-cell is a weak equivalence. Here $\cof$ is the class of cofibrations in \mc{V}.
\end{definition}

\begin{remark}
A strongly left proper monoidal model category is left proper. If a monoidal model category has only cofibrant objects, it is strongly left proper.
\end{remark}

\begin{lemma}\label{gluing lemma}
A monoidal model category \mc{V} is strongly left proper if and only if the following gluing lemma holds: Consider a diagram
\diagr{  D\otimes X \ar[d] & C\otimes X  \ar[d]\ar[l]\ar[r] & S  \ar[d] \\
         D\otimes Y & C\otimes Y  \ar[l]\ar[r] & T }
where all the vertical maps are weak equivalences and the two left hand vertical maps are maps in $\cof\otimes{\rm Ob}(\mc{V})$. Then the induced vertical map on the pushout is a weak equivalence.
\end{lemma}

The proof is analogous to the proof, that left properness is equivalent to the statement, that the pushout is a weak equivalence if the two left hand vertical maps are cofibrations. It is left to the reader.

\begin{remark}\label{comparison}
For a functor \dgrm{X} in $\mc{V}^{\mc{U}}$ and an object $A$ in \mc{U} there is a comparison isomorphism
    $$ \dgrm{X}\otimes R^A\cong\dgrm{X}\circ R^A $$
constructed in the following way:
\begin{align*}
    \dgrm{X}\otimes R^A&\cong\int^{K\in\mc{K}}R^{K\otimes A}\otimes\dgrm{X}(K)\cong\int^{K\in\mc{K}}(R^{K}\circ R^A)\otimes\dgrm{X}(K) \\
                       &\cong\dgrm{X}\circ R^A 
\end{align*}
Here \mc{K} is a defining subcategory for \dgrm{X} and we use, that the monoidal product commutes with \mc{V}-colimits.
\end{remark}

\begin{lemma}[\cite{DRO:enriched} Theorem 4.11]\label{tensoring preserves obj we}
Let \mc{V} be right proper and cofibrantly generated, such that the sources and targets of the maps in $I_{\mc{V}}$ are finitely presentable and that tensoring with them preserves weak equivalences in \mc{V}.
We assume further, that \mc{V} is strongly left proper and satisfies the monoid axiom. 
Let axioms $(7)$ and $(8)$ from \emph{\ref{U3}} be satisfied.
Then tensoring with a projectively cofibrant functor preserves objectwise weak equivalences. 
\end{lemma}

\begin{proof}
Let $\dgrm{X}\to\dgrm{Y}$ be an objectwise weak equivalence. Using the comparison map from \ref{comparison} we have for every $K$ in \mc{U} and $M$ in \mc{M} the following diagram:
\diagr{ \dgrm{X}\otimes(R^K\otimes M) \ar[r]\ar[d]_{\cong} & \dgrm{Y}\otimes(R^K\otimes M) \ar[d]_{\cong} \\
        \dgrm{X}(\mc{V}(K,\frei))\otimes M \ar[r] & \dgrm{Y}(\mc{V}(K,\frei))\otimes M }
The map $\dgrm{X}(\mc{V}(K,\frei))\to\dgrm{Y}(\mc{V}(K,\frei))$ is an objectwise weak equivalence by \ref{LKan preserves we}. If $M$ is a source or target of a map in $I_{\mc{V}}$, the lower horizontal map is a weak equivalence by assumption and so is the upper.

Now let \dgrm{E} be a projectively cofibrant functor. It is a consequence of the small object argument, that \dgrm{E} is a retract of the class $I_{\mc{M}^{\mc{U}}}^{\rm proj}$-cell. Further we know by \ref{we closed under filtered colimits}, that objectwise weak equivalences are closed under transfinite compositions. So we can reduce to the case, where \dgrm{E} is a finite composition of cobase changes of $I_{\mc{M}^{\mc{U}}}^{\rm proj}$. Here we can proceed by induction, where the induction start is explained above. For the induction step we have to look at the following diagram:
\diagr{ \bigsqcup_{m}(R^{A_m}\otimes D_m)\otimes\dgrm{X} \ar[d] & \bigsqcup_{m}(R^{A_m}\otimes C_m)\otimes\dgrm{X} \ar[d]\ar[l]\ar[r] & \dgrm{E}_n\otimes\dgrm{X} \ar[d] \\
        \bigsqcup_{m}(R^{A_m}\otimes C_m)\otimes\dgrm{Y} & \bigsqcup_{m}(R^{A_m}\otimes C_m)\otimes\dgrm{Y} \ar[l]\ar[r] & \dgrm{E}_n\otimes\dgrm{Y}  }
The right vertical map is an objectwise weak equivalence by the induction assumption. So are the other two vertical maps by the previous paragraph. 

The two left hand horizontal maps are not necessarily objectwise cofibrations, but they are coproducts of maps, that are objectwise in the class $\cof\otimes{\rm Ob}(\mc{V})$. Now, pushouts in the functor category are computed objectwise.
So by the gluing lemma \ref{gluing lemma}, which holds due to the strong left properness of \mc{V}, the induced vertical map on the pushout $\dgrm{E}_{n+1}\otimes\dgrm{X}\to\dgrm{E}_{n+1}\otimes\dgrm{Y}$, which is computed objectwise, is an objectwise weak equivalence.
\end{proof}

\begin{lemma}\label{tensoring with proj. cof. preserves L-we}
Let \mc{V} and \mc{U} be as in \emph{\ref{tensoring preserves obj we}}.
Tensoring with a projectively cofibrant object preserves $L$-stable equivalences. 
\end{lemma}

\begin{proof}
We factor the $L$-stable equivalence into an $L$-stably acyclic cofibration followed by an $L$-stably acyclic fibration, which is just an objectwise acyclic fibration. Tensoring with the first map remains an $L$-stable equivalence because of the compatibility result in \ref{V^V symmetric monoidal}. Tensoring with the second map remains an objectwise equivalence by \ref{tensoring preserves obj we}.
\end{proof}

\begin{theorem}\label{monoid axiom}
Let \mc{V} and \mc{U} be as in \emph{\ref{tensoring preserves obj we}}. Then the $L$-stable model structure on $\mc{V}^{\mc{U}}$ satisfies the monoid axiom.
\end{theorem}

\begin{proof}
We first prove, that tensoring a map in $J^{L{\rm -stable}}_{\mc{M}^{\mc{U}}}= J'''\cup J_{\mc{M}^{\mc{U}}}^{\rm ho}$ described in \ref{J^lin} with an arbitrary functor \dgrm{X} in $\mc{V}^{\mc{U}}$ remains an $L$-stable equivalence, and it suffices to do this for $J'''$. Let $\alpha\co\dgrm{A}\to\dgrm{X}$ be a projectively cofibrant replacement for \dgrm{X}. Let $j\co\dgrm{Q}\to\dgrm{R}$ be a map in $J'''$. We have a diagram:
\diagr{ \dgrm{Q}\otimes\dgrm{A} \ar[r]\ar[d] & \dgrm{Q}\otimes\dgrm{X} \ar[d] \\
        \dgrm{R}\otimes\dgrm{A} \ar[r] & \dgrm{R}\otimes\dgrm{Y} } 
Since the source \dgrm{Q} and the target \dgrm{R} of $j$ are projectively cofibrant, the horizontal maps are objectwise weak equivalences by \ref{tensoring preserves obj we}. The left hand vertical map is an $L$-stable equivalence by \ref{tensoring with proj. cof. preserves L-we}, since \dgrm{A} is projectively cofibrant and the claim follows.

Next we have to show, that the cobase change of a map in $J_{\mc{V}^{\mc{U}}}^{L{\rm -stable}}\otimes{\rm Ob}(\mc{V}^{\mc{U}})$ is an $L$-stable equivalence. Let \dgrm{X} be an arbitrary functor in $\mc{V}^{\mc{U}}$ and let $\dgrm{Q}\to\dgrm{R}$ be a map in $J_{\mc{V}^{\mc{U}}}^{L{\rm -stable}}$. Let further $f\co\dgrm{Q}\otimes\dgrm{X}\to\dgrm{Y}$ be some map. We factor $f$ into a projective cofibration $i\co\dgrm{Q}\otimes\dgrm{X}\to\dgrm{E}$ followed by an objectwise acyclic fibration $\dgrm{E}\to\dgrm{Y}$. Then the pushout of $f\otimes\id_{\dgrm{X}}$ along $i$ remains an $L$-stable equivalence because of left properness. 
\diagr{ \dgrm{Q}\otimes\dgrm{X} \ar[r]\ar[d]_{L-\simeq}\ar@{}[dr]|->>>{\pushout} & \dgrm{E} \ar[r]^-{{\rm obj}-\simeq}\ar[d]_-<<{L-\simeq}\ar@{}[dr]|->>>{\pushout} & \dgrm{Y} \ar[d] \\
        \dgrm{R}\otimes\dgrm{X} \ar[r] & \dgrm{P}_1 \ar[r] & \dgrm{P}_2 } 
Now the map $\dgrm{E}\to\dgrm{P}_1$ is objectwise in the class $\cof\otimes{\rm Ob}(\mc{V})$-cell. Since \mc{V} is strongly left proper, the cobase change $\dgrm{P}_1\to\dgrm{P}_2$ is an objectwise weak equivalence. This proves, that $\dgrm{Y}\to\dgrm{P}_2$ is an $L$-stable equivalence.
\end{proof}

\section{$L$-spectra}
\label{section:L-spectra}

In this section we would like to compare $L$-stable functors with $L$-spectra. The first objective is to compare them with symmetric $L$-spectra constructed in \cite{Hovey:general-sym-spec}, because they also have a monoidal product. 
Other references for symmetric spectra are \cite{HSS:sym}, \cite{jar:motsymspec} and \cite{schwede:book}.
We were not able to get a direct Quillen equivalence, but rather a zig-zag. To get better results in certain cases we also compare our $L$-stable functors to Bousfield-Friedlander-$L$-spectra, also constructed in \cite{Hovey:general-sym-spec}.


\begin{definition}
Let $\sigsph_L$ be the \mc{V}-category with the natural numbers $n\ge 0$ as objects and as morphisms
    $$ \sigsph_L(n,m):=\left\{ \begin{array}{cl}
                                    L^{m-n}\otimes_{\Sigma_{m-n}}\Sigma_m &,\text{ for }m\ge n \\
                                     \ast                 &,\text{ else }
                            \end{array}\right. $$
where $\Sigma_k$ acts on $L^k$ by permuting the tensor factors. Of course, we set $L^0=S$. A {\it symmetric spectrum} in \mc{M} is a \mc{V}-functor $X\co\sigsph_L\to\mc{M}$.
We denote the category of symmetric $L$-spectra in \mc{M} by $\symsp(\mc{M},L)$. 
\end{definition} 

There is an inclusion $j\co \sigsph_L\to\vsph_L$ of categories enriched over \mc{V}, but this inclusion is not full. Hence the counit $\LKan_jj^*\to\id$ of the induced \mc{V}-Quillen pair 
\begin{equation}\label{L-spectra Quillen equiv}
    \LKan_j\co\symsp(\mc{M},L)=\mc{M}^{\sigsph_L}\rightleftarrows\mc{M}^{\vsph_L}:\!j^* 
\end{equation}
is not an isomorphism.

\begin{definition}
The {\it evaluation functor} $\Ev_n\co\symsp(\mc{M},L)\to\mc{M}$ sending $X$ to $X_n$ has a left adjoint
    $$ F_n\co\mc{M}\to\symsp(\mc{M},L). $$
For example, $F_0$ is given by $M\mapsto (M,M\otimes L, M\otimes L^{\otimes 2},...)=F_0M.$
\end{definition} 

\begin{lemma}\label{Quillen sym. mon.}
For $\mc{M}=\mc{V}$ the functor $\LKan_j$ in \emph{(\ref{L-spectra Quillen equiv})} is strictly symmetric monoidal, the right adjoint $j^*$ is lax symmetric monoidal. 
\end{lemma} 

\begin{proof}
We observe first, that every symmetric $L$-spectrum $E$ over \mc{V} can be written as a colimit
    $$ E\cong\colim_{i\in I}F_{k_i}V_i, $$
where $V_i$ is an object of \mc{V}. We can compute:
    $$ \LKan_jF_kV\cong R^{L^k}\otimes V $$
The monoidal product on symmetric $L$-spectra is defined by extending the isomorphism
    $$ F_{k}V\otimes F_{\ell}W\cong F_{k+\ell}(V\otimes W),$$
while the monoidal product on small functors is defined by extending the isomorphism
    $$ (R^{A}\otimes V)\otimes (R^B\otimes W)\cong R^{A\otimes B}\otimes(V\otimes W).$$
The claim easily follows.
\end{proof}

\begin{definition}
The $L$-stable model structure on $\symsp(\mc{M},L)$ is obtained by localizing the level model structure with respect to the set of maps 
    $$ F_{n+1}(C\otimes L)\to F_n(C), $$
where $C$ runs through the sources and targets of the generating set of cofibrations $I_{\mc{M}}$ of \mc{M}. For details see \cite[8.7]{Hovey:general-sym-spec}. Here we use left properness of \mc{M} to ensure, that all our $C$'s are cofibrant.
\end{definition} 

\begin{theorem}[Hovey]\label{Hovey,sym} 
Let \mc{C} be a left proper cellular closed symmetric monoidal model category, let \mc{D} be a left proper cellular \mc{C}-model category and let $L$ be a cofibrant \mc{C}-small object. Suppose also, that the source and target of all maps in $I_{\mc{C}}$ and $I_{\mc{D}}$ are cofibrant. Then on $\symsp(\mc{D},L)$ there exists an $L$-stable model category,  $\symsp(\mc{C},L)$ with its $L$-stable model structure becomes a closed symmetric monoidal model category and $\symsp(\mc{D},L)$ is a $\symsp(\mc{C},L)$-model category.  
\end{theorem} 

The previous theorem is proved in \cite[8.8]{Hovey:general-sym-spec} and \cite[8.11]{Hovey:general-sym-spec}. There is no assumption on the cyclic permutation on $L\otimes L\otimes L$. The cofibrancy assumptions seem awkward.
\begin{question}\label{cofibrancy}
In a left proper model category a map is a fibration (resp. acyclic fibration) if and only if it satisfies the right lifting property with respect to all acyclic cofibrations (resp. cofibrations) between cofibrant objects. So can one always obtain from a given set of generators for the (acyclic) cofibrations a generating set, whose sources and targets are cofibrant? 
\end{question}

\begin{definition}
There are mutually inverse isomorphism of categories
\diagr{  \symsp(\mc{M},L)^{\mc{U}}\ar@<2pt>[r]^{I_1} &  \symsp(\mc{M}^{\mc{U}},L), \ar@<2pt>[l]^{I_2} }
which merely change the priority of variables.
\end{definition} 

\begin{lemma}
Suppose that \mc{V} and \mc{M} satisfy in addition to our usual assumptions from section \emph{\ref{section:fixing}} the conditions of Hovey's theorem \emph{\ref{Hovey,sym}}. 
Suppose also, that \mc{U} is small.
We give $\symsp(\mc{M},L)^{\mc{U}}$ the $L$-stable model structure over the $L$-stable model structure of $\symsp(\mc{M},L)$ and $\symsp(\mc{M}^{\mc{U}},L)$ the $L$-stable model structure over the $L$-stable model structure of $\mc{M}^{\mc{U}}$. Then the isomorphisms $I_1$ and $I_2$ in \emph{(\ref{L-spectra Quillen equiv})} are isomorphisms of model structures.
\end{lemma} 

This result permits us to identify the categories $\symsp(\mc{M},L)^{\mc{U}}$ and $\symsp(\mc{M}^{\mc{U}},L)$ with their $L$-stable model structure.

\begin{proof}
Using the fact, that for a functor $\dgrm{X}\co\mc{U}\to\mc{V}$, an object $M$ in \mc{M} and any $n\ge 0$ there is a natural isomorphism
    $$ F_n(\dgrm{X}\otimes M)\cong \dgrm{X}\otimes F_nM ,$$
it is not difficult to check, that the generating sets of cofibrations and acyclic cofibrations of the two model structure are mapped to each other. 
\end{proof}

It is proved in \cite[section 8]{Hovey:general-sym-spec}, that $(F_0,\Ev_0)$ is a \mc{V}-Quillen adjunction, if the source category is already an $L$-stable model structure. It follows, that
\begin{equation}\label{Quillen pair1}
    F_0\co\mc{M}^{\mc{U}}\rightleftarrows\symsp(\mc{M}^{\mc{U}},L):\!\Ev_0
\end{equation}
is a \mc{V}-Quillen equivalence, provided we take the $L$-stable model structures everywhere.

Now we restrict to $\vsph_L$ as source category. We remind the reader, that in this case the $L$-stable model structure on $\mc{M}^{\vsph_L}$ is cofibrantly generated.  
\begin{theorem}\label{Quillen pair3}
Let \mc{V} be a pointed, right proper, strongly left proper, cellular, locally presentable, closed symmetric monoidal model category. Let \mc{M} be a proper, cellular, locally presentable \mc{V}-model category. Let $L$ be cofibrant and \mc{V}-small. Suppose, that source and target of all maps in $I_{\mc{V}}$ and $I_{\mc{M}}$ are cofibrant.
Then there is a zig-zag in $\ulvmod$ of \mc{V}-Quillen equivalences
\diagr{   \symsp(\mc{M},L)^{\vsph_L}\ar@<2pt>[d]^{\rho}\ar@<2pt>[rr]^{I_1} & & \symsp(\mc{M}^{\vsph_L},L) \ar@<-2pt>[d]_{\Ev_0}\ar@<2pt>[ll]^{I_2}   \\
         \symsp(\mc{M},L) \ar@<2pt>[u]^{\lambda}   & & \mc{M}^{\vsph_L} \ar@<-2pt>[u]_{F_0}  }
which is functorial in \mc{M}. For $\mc{V}=\mc{M}$ these Quillen equivalences are symmetric monoidal. 
\end{theorem}

\begin{proof}
We observed already, that the pair $(F_0,\Ev_0)$ is a \mc{V}-Quillen equivalences and that the pair $(I_1,I_2)$ is an isomorphism of model structures. 
But also $(\lambda,\rho)$ is a Quillen equivalence by \ref{S-eval}.
\end{proof}

It is tempting to complete the above diagram in the following way:
\diag{   \symsp(\mc{M},L)^{\vsph_L}\ar@<2pt>[d]^{\rho}\ar@<2pt>[rr]^{I_1} & & \symsp(\mc{M}^{\vsph_L},L) \ar@<-2pt>[d]_{\Ev_0}\ar@<2pt>[ll]^{I_2}   \\
         \symsp(\mc{M},L) \ar@<2pt>[rr]^-{\LKan_j}\ar@<2pt>[u]^{\lambda}   & & \mc{M}^{\vsph_L} \ar@<2pt>[ll]^-{j^*} \ar@<-2pt>[u]_{F_0}  }{bispectrum}
Here the lower horizontal maps are the ones from (\ref{L-spectra Quillen equiv}).
Unfortunately this diagram is not commutative. It would be possible to prove, that the lower horizontal maps are a \mc{V}-Quillen equivalence, if we had an answer to the following question.
\begin{question}\label{stabil2}
Is the unit map $E\to j^*\LKan_jE$ in $\mc{M}^{\sigsph_L}$ a stable equivalence of $L$-spectra, if $E$ is cofibrant in $\symsp(\mc{M},L)$?
\end{question}

We have no doubt, that in all reasonable cases the answer to this question is affirmative. The problem is, that $L$-stable equivalences and $L$-stable fibrations of symmetric $L$-spectra are rather difficult to characterize.
We point out however, that in the special cases or our examples in the next section \ref{section:examples} we get better results. The reason is, that we can use ordinary Bousfield-Friedlander-$L$-spectra as stepping stone for the comparison, see remark \ref{spectra-comparison}. 

\begin{definition}
Let $\sph_L$ be the \mc{V}-category with objects the natural numbers $n\in\mathbbm{N}$ and morphisms
    $$ \sph_L(n,m)=\left\{\begin{array}{cl}
                              L^{\otimes m-n} &,\text{ for } m\ge n \\
                                \ast          &,\text{ for } m< n. 
                          \end{array}\right.$$
It comes with obvious functors $i\co\sph_L\to\sigsph_L$ and $k\co\sph_L\to\vsph_L$. The category $\Sp(\mc{M},L)$ of \mc{V}-functors from $\sph_L$ to \mc{M} is the category of {\it Bousfield-Friedlander spectra}. For details see \cite{BF:gamma} and \cite{Hovey:general-sym-spec}.
\end{definition}

First we recall theorems 4.12, 4.14 and 6.5 from \cite{Hovey:general-sym-spec}. 
\begin{lemma}[Hovey]\label{Hovey,BF}
\begin{punkt}
   \item 
Let \mc{M} be a pointed proper locally presentable almost finitely generated model category. Let $L$ in \mc{V} be cofibrant and \mc{V}-small. Then the $L$-stable model structure makes $\Sp(\mc{M},L)$ into a proper \mc{V}-model category.
   \item 
The map $X\to PX$ is an $L$-stable equivalence for all $X$ in $\Sp(\mc{M},L)$.
   \item
A map $X\to Y$ in $\Sp(\mc{M},L)$ is an $L$-stable equivalence if and only if $PX\to PX$ is a level equivalence.
   \item
A map $X\to Y$ in $\Sp(\mc{M},L)$ is an $L$-stable fibration if and only if it is a level fibration, such that the diagram
\diagr{ X \ar[r]\ar[d] & PX \ar[d] \\ Y \ar[r] & PY}
is a homotopy pullback square in the level projective model structure.
   \item
The $L$-stable model structure gives $\Sp(\mc{M},L)$ the structure of a \mc{V}-model category.
\end{punkt}
\end{lemma}

\begin{lemma}\label{k detects we}
The functor $k^*\co\mc{M}^{\vsph_L}\to\Sp(\mc{M},L)$ preserves and detects $L$-stable equivalences.
\end{lemma}

\begin{proof}
Follows directly from \ref{Hovey,BF}(iii).
\end{proof}

We chose $\fib$ to be a small fibrant replacement functor in \mc{V} or \mc{M} and we abbreviated by $(\frei)^h$ the assignment $\dgrm{X}\mapsto\fib\circ\dgrm{X}\circ\fib$. 
\begin{lemma}[\cite{DRO:enriched} Cor. 7.4.]\label{unit for BF spectra}
We take for granted our standard assumptions from section \emph{\ref{section:fixing}}.
Suppose that \mc{M} is left proper and that the cyclic permutation on $L\otimes L\otimes L$ is \mc{V}-left homotopic to the identity.
Then the unit map 
    $$E\to k^*(\LKan_kE)$$ 
in $\Sp(\mc{M},L)$ is a stable equivalence of Bousfield-Friedlander $L$-spectra for every cofibrant $E$.
\end{lemma}

\begin{proof}
First one proves the claim for the special case $E=\sph_L(n,\frei)$. In the diagram
\diagr{ \sph_L(n,\frei) \ar[r]^-1\ar[drr]_-4 & k^*(\LKan_k\sph_L(n,\frei)) \ar[r]^-2 & k^*(\LKan_k\sph_L(n,\frei))^h \ar[d]^3 \\
          & & P(R^{L^n}) }
we compute $k^*(\LKan_k\sph_L(n,\frei))^h\cong(R^{L^n})^h$ and conclude, that map 3 is an $L$-stable equivalence. 
Map 2 is an objectwise equivalence.
But also the composition
    $$ \sph_L(n,\frei)\otimes L^n\to\sph_L(0,\frei)\to(\sph_L(0,\frei))^h\to P(\sph_L(0,\frei)) $$
is an $L$-stable equivalence. Since tensoring with $L$ is a Quillen equivalence in the $L$-stable model structure, the adjoint map
    $$ \sph_L(n,\frei)\to \mc{V}(L^n,P(\sph_L(0,\frei))\cong P(R^{L^n})  $$
is an $L$-stable equivalence. This is map 4 of the diagram. Here we are using the cyclic permutation assumption, because \cite[10.3]{Hovey:general-sym-spec} shows, that $\frei\otimes L$ is a Quillen equivalence only under this condition.
By 2-out-of-3 this proves the special case, map 1 is an $L$-stable equivalence. 

For general cofibrant $E$ we know, that $E$ is a retract of the class
    $$I_{\Sp(\mc{M},L)}^{L\text{-stable}}\text{-cell}=(\bigcup_{n\ge 0}\sph_L(n,\frei)\otimes I_{\mc{M}})\text{-cell}, $$
where $I_{\Sp(\mc{M},L)}^{L\text{-stable}}$ is the class of generating cofibrations for the $L$-stable model structure on $\Sp(\mc{M},L)$ by \cite[1.8]{Hovey:general-sym-spec}. Here, for $i\in I_{\mc{M}}$, we have:
    $$ \sph_L(n,\frei)\otimes i= \ol{F}_ni, $$
where $\ol{F}_n\co\mc{M}\to\Sp(\mc{M},L)$ is the \mc{V}-left adjoint to the $n$-th evaluation functor $\Ev_n\co\Sp(\mc{M},L)\to\mc{M}$. Since $L$-stable equivalences of $L$-spectra are closed under sequential colimits, we can prove the claim by induction along pushouts. We use, that $j^*$ and $\LKan_j$ commute with colimits. We obtain diagrams
\diagr{ \bigsqcup_m\sph_L(n_m,\frei)\otimes D_m \ar[d] & \bigsqcup_m\sph_L(n_m,\frei)\otimes C_m \ar[d]\ar[l]\ar[r] & E \ar[d] \\
        \bigsqcup_m R^{L^n}\otimes D_m & \bigsqcup_m R^{L^n}\otimes C_m \ar[r]\ar[l] & j^*(\LKan_j E) }
where the maps $C_m\to D_m$ are in $I_{\mc{M}}$. All the vertical maps are $L$-stable equivalences. Since \mc{M} is left proper, we can assume, that all $C_m$'s and $D_m$'s are cofibrant. 
Hence the two left hand horizontal maps are cofibrations. Again by left properness we can conclude, that the pushout is an $L$-stable equivalence. 
\end{proof}

\begin{theorem}
Under the assumptions from lemma \emph{\ref{unit for BF spectra}} the pair of functors
    $$k^*\co\mc{M}^{\vsph_L}\to\Sp(\mc{M},L):\!\LKan_k$$ 
is a \mc{V}-Quillen equivalence.
\end{theorem}

\begin{proof}
This follows directly from \ref{k detects we} and \ref{unit for BF spectra}.
\end{proof}

\begin{remark}\label{spectra-comparison}
We have a commutative diagram of \mc{V}-Quillen pairs:
\diag{   \symsp(\mc{M},L)\cong\mc{M}^{\sigsph_L} \ar@<2pt>[rr]^-{\LKan_j}\ar@<2pt>[rd] & & \mc{M}^{\vsph_L} \ar@<2pt>[ll]\ar@<2pt>[dl] \\
          & \Sp(\mc{M},L)\cong\mc{M}^{\sph_L} \ar@<2pt>[ul]^-{\LKan_i} \ar@<2pt>[ur]^-{\LKan_k}  & }{comparison of spectra}
This shows, that under the assumptions from lemma \ref{unit for BF spectra} $L$-stable functors are \mc{V}-Quillen equivalent to symmetric $L$-spectra if and only if Bousfield-Friedlander-$L$-spectra are \mc{V}-Quillen equivalent to symmetric $L$-spectra. See section \ref{section:examples} for applications. Of course, this is cheating and we believe that the pair $(\LKan_j,j^*)$ should be a Quillen equivalence under much more general conditions. In particular, the cyclic permutation condition should not play a role.
\end{remark}

\section{Examples}
\label{section:examples}

Let us recall again the assumptions, under which the $L$-stable model structure on $\mc{M}^{\mc{U}}$ exists. 

\begin{itemize}
   \item
The category \mc{V} has to be a locally presentable cofibrantly generated symmetric monoidal model category equipped with a functorial small fibrant replacements.
   \item
The object $L$ in \mc{V} has to be \mc{V}-small and cofibrant. If the unit $S$ of \mc{V} is small, every small object in \mc{V} is \mc{V}-small. 
   \item
The category \mc{M} should be a right proper locally presentable cofibrantly generated \mc{V}-model category with a chosen functorial small fibrant replacements. 
   \item
For the $L$-stable model structure to exist on $\mc{M}^{\mc{U}}$ the category \mc{U} has to satisfy the axioms \ref{U1}. 
   \item
In order to have a closed symmetric monoidal structure  on $\mc{V}^{\mc{U}}$ we need \ref{U2}. 
   \item
Right properness of \mc{V} and properness of \mc{M} imply all the compatibility results. 
   \item
Finally for the monoid axiom in $\mc{V}^{\mc{U}}$ we need, that \mc{V} is strongly left proper, satisfies the monoid axiom and sources and targets of $I_{\mc{V}}$ are \mc{V}-finitely presentable. Additionally \mc{U} should satisfy \ref{U3}. If \mc{V} has only cofibrant objects, it is strongly left proper and satisfies the monoid axiom.
\end{itemize}

For our $L$-stabilization we considered the category $\vsph_L$ as source category \mc{U}. It satisfies axioms \ref{U1}, \ref{U2} and, if $L$ is \mc{V}-finitely presentable, (\ref{8}). Axiom (\ref{7}) is obviously highly sensitive to the choice of $L$.

\begin{example}
Let $\mc{V}=\mc{S}_*$ pointed simplicial sets and $L=S^1$. In this case we know by \cite[Theorem 4.2.5]{HSS:sym}, that the diagram (\ref{comparison of spectra}) consists of Quillen equivalences. 
So the homotopy category associated to the category $\mc{S}_*^{\sph_{S^1}^{\mc{S}_*}}$ with the $S^1$-stable model structure is the classical stable homotopy category with the correct monoidal product. This was first proved in \cite{Lydakis}. The monoid axiom holds for $S^1$-stable functors, and was first proved in \cite[lemma 6.30]{DRO:enriched}. For symmetric spectra this was proved in \cite[Theorem 5.4.1]{HSS:sym}.
\end{example}

\begin{example}
Let $S$ be a noetherian scheme of finite dimension.
Let $\mc{V}$ be the category of $\mathbbm{A}^1$-local motivic spaces over a scheme $S$. There are several ways to construct this model category, all of them Quillen equivalent to each other, and all serve well for our purpose. See \cite{morel-voev:mot-hom}, \cite{jar:motsymspec} or \cite{DRO:mot}.
Let $L$ be the smash product $\mathbbm{P}^1$ of the simplicial circle $S^1$ with the Tate circle $\mathbbm{A}^1-\{0\}$. This object is cofibrant and and \mc{V}-finitely presentable. It satisfies the cyclic permutation condition \cite[3.13]{jar:motsymspec}. By \cite[Theorem 4.31]{jar:motsymspec} we know, that (\ref{comparison of spectra}) is a diagram of \mc{V}-Quillen equivalences. Hence $\mathbbm{P}^1$-stable functors supply a model for the stable motivic homotopy category together with the correct smash product. This was first observed in \cite{DRO:enriched} and \cite{DRO:mot}.

More generally, Jardine has constructed $L$-stable spectra and symmetric spectra over motivic spaces for $L=S^1\wedge G$, where $G$ is a compact object, see \cite[2.2]{jar:motsymspec}. 
They are shown to be Quillen equivalent in \cite[Theorem 4.31]{jar:motsymspec}. 
Compact object in Jardine's sense are \mc{V}-finitely presentable, and all objects are cofibrant in his setting. So our theory applies and again it follows, that our $L$-stable functors are Quillen equivalent to them. 
\end{example}

\begin{example}
In \cite[section 9]{DRO:enriched} it is shown, that $G$-equivariant stable homotopy theory can be described by this setup. 
\end{example}


\end{document}